\documentclass[reqno,a4paper]{amsart}
\usepackage{enumerate}
\usepackage{amsfonts,amssymb,amsmath,amsthm}
\usepackage[backend=biber, style=alphabetic, sorting=nyt, maxbibnames=99]{biblatex}
\usepackage{epsfig}
\usepackage{graphics}
\usepackage[normalem]{ulem}
\usepackage{color}
\usepackage{xypic}
\usepackage{version}
\usepackage{caption}
\usepackage{subcaption}
\usepackage{hyperref}
\usepackage{tikz}
\usetikzlibrary{calc}
\usepackage{tikz-3dplot}
\usepackage{pgfplots}

\numberwithin{equation}{section}

\definecolor{OrangeRed}{cmyk}{0,0.6,1,0}            
\definecolor{DarkBlue}{cmyk}{1,1,0,0.20}
\definecolor{DarkGreen}{cmyk}{1,0,0.6,0.2}
\definecolor{myblue}{rgb}{0.66,0.78,1.00}
\definecolor{Violet}{cmyk}{0.79,0.88,0,0}
\definecolor{Lavender}{cmyk}{0,0.48,0,0}

\newtheorem{thm}{Theorem}[section]
\newtheorem{theorem}[thm]{Theorem}
\newtheorem{main theorem}[thm]{Main Theorem}
\newtheorem{corollary}[thm]{Corollary}
\newtheorem*{main}{Main Theorem}
\newtheorem{lemma}[thm]{Lemma}

\theoremstyle{definition}

\newtheorem{definition}[thm]{Definition}
\newtheorem{remark}[thm]{Remark}

\newtheorem{example}[thm]{Example}

\DeclareMathOperator{\Int}{int}
\DeclareMathOperator{\Conv}{Conv}
\newcommand{\be}{\begin{equation}}
\newcommand{\ee}{\end{equation}}
\DeclareMathOperator{\weight}{wt}
\newcommand{\PPr}{\mathbb{P}}

\title[Uniqueness of the Gibbs measure for the $4$-state Potts model]{Uniqueness of the Gibbs measure for the $4$-state anti-ferromagnetic Potts model on the regular tree}

\author{David de Boer}
\author{Pjotr Buys}
\thanks{$\ddagger$ DdB and PB are funded by the Netherlands Organisation of Scientific Research (NWO): 613.001.851}
\author{Guus Regts}
\thanks{$\diamond$ Funded by the Netherlands Organisation of Scientific Research (NWO): VI.Vidi.193.068}

\date{\today}

\address[David de Boer, Pjotr Buys, Guus Regts]{Korteweg de Vries Institute for Mathematics, University of Amsterdam. P.O. Box 94248  
1090 GE Amsterdam  
The Netherlands}
\email{\{daviddeboer2795,pjotr.buys,guusregts\}@gmail.com}

\begin{document}



\begin{abstract}
We show that the $4$-state anti-ferromagnetic Potts model with interaction parameter $w\in(0,1)$ on the infinite $(d+1)$-regular tree has a unique Gibbs measure if $w\geq 1-\frac{4}{d+1}$ for all $d\geq 4$. This is tight since it is known that there are multiple Gibbs measures when $0\leq w<1-\frac{4}{d+1}$ and $d\geq 4$.

We moreover give a new proof of the uniqueness of the Gibbs measure for the $3$-state Potts model on the $(d+1)$-regular tree for $w\geq 1-\frac{3}{d+1}$ when $d\geq 3$ and for $w\in (0,1)$ when $d=2$.
\end{abstract}

\maketitle

\begin{section}{Introduction}\label{sec:intro}

The Potts model, originally invented to study ferromagnetism~\cite{potts}, is a model from statistical physics; it also plays a central role in probability theory, combinatorics and computer science. 

Let $G = (V,E)$ be a finite graph. The Potts model on the graph $G$ has two parameters, a number of \emph{states} $q \in \mathbb{Z}_{\geq2}$ and an interaction parameter $w\geq 0$. The case $q=2$ is known as the Ising model.
A \emph{configuration} is a map $\sigma:V \to [q]:=\{1,\dots,q\}$. 
Associated with such a configuration is a \emph{weight}\footnote{In case $w=0$ we implicitly assume that there is at least one configuration of non-zero weight, i.e., a proper coloring.} $w^{m(\sigma)}$, where $m(\sigma)$ is the number of monochromatic edges in the configuration $\sigma$. 
The $q$-state \emph{partition function of the Potts model} is the sum of the weights over all configurations; we denote it as
$
Z(G, q, w) = \sum_{ \sigma:V \rightarrow [q]} w^{m(\sigma)}.
$
In statistical physics one has $w=e^{kJ/T}$, with $J$ being an interaction parameter, $k$ the Boltzmann constant and $T$ the temperature.
We write $Z(G)$ to keep notation short.

The Gibbs measure is the probability measure $\PPr_G[\cdot]$ on the set of configurations of $G=(V,E)$, where the probability of a random configuration\footnote{We use the convention to denote random variables in boldface.} $\bf{\Phi}$ being equal to a given configuration $\phi:V\to [q]$ is proportional to the weight of $\phi$:
\begin{equation}\label{eq:gibbs finite}
\PPr_{G}[{\mathbf \Phi}=\phi] = \frac{w^{m(\sigma)}}{Z(G)}.
\end{equation}

The Potts model is said to be \emph{ferromagnetic} if $w >1$ and \emph{anti-ferromagnetic} if $w < 1$. 
The ferromagnetic Potts model favors configurations with a large number of monochromatic edges, while the anti-ferromagnetic Potts model favors configurations with a small number of monochromatic edges, i.e., configurations that are `close' to proper colorings.

In statistical physics, models like the Potts model are typically considered on infinite graphs such as $\mathbb{Z}^d$ or the Bethe lattice $\mathbb{T}_d$, also known as the infinite $(d+1)$-regular tree.
One can extend the notion of Gibbs measures for finite graphs to this infinite setting (see below for more details).
The Gibbs measure on a finite graph is clearly unique. 
For infinite graphs, depending on the underlying parameter $w$, there may however be multiple Gibbs measures.
The transition from having a unique Gibbs measure to multiple Gibbs measures is referred to as a \emph{phase transition} in statistical physics~\cite{friedli2017} and it is an important problem to determine when this happens in terms of the underlying parameters of the model. 
Moreover, for several $2$-state models, the uniqueness region and the transition from uniqueness to non-uniqueness of the Gibbs measure on $\mathbb{T}_d$ have been connected to the tractability of approximately computing partition functions of these models. 
See e.g.~\cite{weitz06,sly10,slysun14,gal16,sst14,licor}. 
In the case of the anti-ferromagnetic Potts model it is known that in the uniqueness regime on $\mathbb{T}_d$ there is an efficient algorithm to approximately compute the partition function and sample from the Gibbs measure on random $(d+1)$-regular graphs~\cite{blanca2020}. 
See also~\cite{eft2020} for related results on Erd\H{o}s-R\'enyi random graphs without any assumption on uniqueness.
It is moreover expected that the uniqueness to non-uniqueness transition for the anti-ferromagnetic Potts model says something about the tractability of approximating the partition function for the entire class of bounded degree graphs. In particular, approximating the partition function of the Potts model is NP-hard on graphs of maximum degree $d+1$ when $0<w<1 - \frac{q}{d+1}$~\cite{gal15} (for even $q$). 
It is a major open problem to determine whether there exist efficient algorithms for all $w\in(1-\frac{q}{d+1},1]$.

In the present paper we consider the problem of determining when the anti-ferromagnetic Potts model on the infinite $(d+1)$-regular tree has a unique Gibbs measure. 
Before stating our main result, we first give a formal definition of Gibbs measures on the $(d+1)$-regular tree.

\subsection*{Gibbs measures, uniqueness and main result}
We follow Brightwell and Winkler~\cite{brigt99,bright02} to introduce the notion of Gibbs measures on $\mathbb{T}_d$, see also~\cite{rozikov2013gibbs,friedli2017} for more details and background.

Throughout we fix a degree $d\geq 2$ and an integer $q\geq 2$. 
We denote the vertex set of $\mathbb{T}_d$ by $V_d$
and we denote the space of all configurations $\{\psi:V_d\to [q]\}$ by $\Omega_{q,d}$.
For a set $U\subset V_d$ we denote by $\partial U$ the set of vertices in $U$ that are adjacent to some vertex in $V_d\setminus U$. We refer to $\partial U$ as the \emph{boundary} of $U$. We denote by $U^{\circ}:=U\setminus \partial U$ the \emph{interior} of $U$. For $\psi\in \Omega_{q,d}$ and $U\subset V_d$ we denote the restriction of $\psi$ to $U$ by $\psi\!\restriction_U$.
\begin{definition}[Gibbs measure]
We equip $\Omega_{q,d}$ with the sigma algebra generated by sets of the form $\{\psi\in \Omega_{q,d}\mid \psi\!\restriction_U=\phi\}$ where $U\subset V_d$ is a finite set and $\phi:U\to[q]$ a fixed coloring of $U$.
A probability measure $\mu$ on $\Omega_{q,d}$ is called a \emph{Gibbs measure} if for any finite set $U\subset V_d$ and $\mu$-almost every $\phi\in \Omega_{q,d}$, we have
\begin{equation}\label{eq:def Gibbs}
\PPr_{{\bf \Phi}\sim \mu}[{\bf \Phi}\!\restriction_{U}=\phi\!\restriction_{U} \mid {\bf {\bf \Phi}}\!\restriction_{V_d\setminus U^{\circ}}= \phi\!\restriction_{V_d\setminus U^\circ}]=
\PPr_{U}[{\bf \Phi}\!\restriction_{U}=\phi\!\restriction_U \big | {\bf \Phi}\!\restriction_{\partial U}=\phi\!\restriction_{\partial U}],
\end{equation}
where the second probability $\PPr_{U}$ denotes the probability of seeing configuration $\phi$ on the finite graph $\mathbb{T}_d[U]$ induced by $U$ conditioned on the event of being equal to $\phi$ on $\partial U$. This latter probability is obtained by dividing the weight of ${\bf \Phi}\!\restriction_U$ by the sum of the weights of all colorings of $U$ that agree with $\phi$ on $\partial U$, cf.~\eqref{eq:gibbs finite}.
\end{definition}
\begin{remark}
Note that the conditional probability on the left-hand side of \eqref{eq:def Gibbs} cannot be computed using the standard formula for conditional probabilities, as we in general condition on an event of measure zero. Therefore the formalism of conditional expectations should be used to evaluate this conditional probability. See~\cite{friedli2017} for more details.
\end{remark}

By a compactness argument one can show that there always is at least one Gibbs measure on $\Omega_{q,d}$ cf.~\cite{friedli2017,brigt99}. 
The question of whether there is a unique Gibbs measure can be reformulated in terms of a certain decay of correlations. To do so we require some definitions.
We denote by $\mathbb{T}^n_{d}$ the finite tree obtained from $\mathbb{T}_d$ by fixing a root vertex $r_d$, deleting all vertices at distance more than $n$ from the root, deleting one of the neighbors of $r_d$ and keeping the connected component containing $r_d$.
We denote the set of leaves of $\mathbb{T}^n_{d}$ by $\Lambda_{n,d}$, except when $n=0$, in which case we let $\Lambda_{d,0}=\{r_d\}$.
We omit the reference to $d$ when this is clear from the context.
The next lemma reformulates uniqueness of the Gibbs measure in terms of the dependence on the distribution of the colors of the root vertex on the coloring of the leaves. 

    \begin{lemma}\label{def:uniqueness}
    The $q$-state Potts model with parameter $w\geq 0$ on the infinite $(d+1)$-regular tree has a \emph{unique} Gibbs measure if and only if for all colors $c \in [q]$ it holds that 
    \begin{equation}\label{equation:uniqueness}
        \limsup_{n \to \infty} \max_{\tau: \Lambda_{n,d} \to [q]} \bigg\vert \PPr_{\mathbb{T}^n_{d}} [ {\bf \Phi}(r_d) = c \ \vert\ {\bf\Phi}\!\restriction_{ \Lambda_{n,d}} = \tau] - \frac{1}{q}\bigg\vert = 0.
    \end{equation}
    \end{lemma}
While this result is well known we will provide a proof for convenience of the reader in Appendix~\ref{app:proof} based on Brightwel and Winkler's proof~\cite[Theorem 3.3]{bright02} for the case $w=0$.
We moreover note that~\eqref{equation:uniqueness} is the property of uniqueness used in algorithmic applications~\cite{blanca2020}.

Define $w_c:=\max\{0,1-\frac{q}{d+1}\}$.
It is a folklore conjecture that this Gibbs measure is unique if and only if $w\geq w_c$ (if $q=d+1$, the inequality should be read as a strict inequality). Non-uniqueness for $w<w_c$ has been known for a long time~\cite{Peruggi1,Peruggi2}

For $q>d+1$ and thus $w_c=0$, a Gibbs measure is supported on proper $q$-colorings and in this case the conjecture has been shown to be true by Jonasson~\cite{jon02}. 
For the case $q=3$ and $d\geq 2$ and the case $q=4$ and $d=4$ this has recently been proved by Galanis, Goldberg and Yang~\cite{PottsLeslie3}.
Our main result confirms this conjecture for $q=4$ and $d\geq 4$.
Very recently Bencs together with the authors of the present paper~\cite{bencs2022uniqueness} have confirmed this conjecture for all $q\geq 5$ provided $d$ is large enough.

\begin{main} \label{thm:main}
Let $d\in \mathbb{N}_{\geq 4}$.
Then the $4$-state anti-ferromagnetic Potts model on $\mathbb{T}_d$ has a unique Gibbs measure if and only if $w\geq1-\frac{4}{d+1}$.
\end{main}
Our proof of this result follows a different approach than the one taken in~\cite{PottsLeslie3}, which heavily relies on rigorous (but not easily verifiable) computer calculations. 
In particular, our approach allows us to recover the results from~\cite{PottsLeslie3}, thereby removing the need for these computer calculations. See Theorem~\ref{thm:precisemaintheorem} below for the full statement of what we prove with our approach. 
\\

\noindent{\bf Organization.}
In the next section we discuss our approach towards proving our main theorem arriving at a geometric condition for uniqueness that we check in Section~\ref{sec:proof main} to prove our main theorem, deferring the verification of a crucial inequality to Section~\ref{sec: Inequality section}. Finally, in Section~\ref{sec:conclude} we finish with some concluding remarks and open questions.

\end{section}

\begin{section}{Approach and setup}

Our main goal in this section is to derive a geometric condition for ratios of probabilities that implies uniqueness of the Gibbs measure on the $(d+1)$-regular tree. This condition will then be verified in the following sections. Along the way we will comment on how our approach relates to the approach of Galanis, Goldberg and Yang~\cite{PottsLeslie3}.

\begin{subsection}{Ratios of probabilities and the tree recursion}
Instead of working directly with the probabilities we work with ratios of probabilities just as in~\cite{PottsLeslie3}. 

Let us introduce a few concepts to facilitate the discussion. 
Fix $n,d\in \mathbb{N}$ and write $\mathbb{T}^n_{d}=(V,E)$.
Let $\tau:\Lambda_{n,d}\to [q]$. This will be called a \emph{boundary condition}. 
We denote by
\begin{equation}\label{eq:boundary pf}
Z_{\tau}(\mathbb{T}^n_{d}) = \sum_{\substack{ \sigma:V \rightarrow [q]\\ \sigma\restriction_{\!\Lambda_{n,d}} = \tau}} w^{m(\sigma)},
\end{equation}
the restricted partition function.
For $i \in [q]$ we denote by $Z_{i,\tau}(\mathbb{T}^n_{d})$ the sum \eqref{eq:boundary pf} restricted to those $\sigma$ that associate color $i$ to the root vertex.
We define the \emph{ratio}
\begin{equation}\label{eq:ratio}
R_{i,\tau}(\mathbb{T}^n_{d}) =  \frac{Z_{i,\tau}(\mathbb{T}^n_{d})}{Z_{q,\tau}(\mathbb{T}^n_{d})}.
\end{equation}
Note that $R_{q,\tau}(\mathbb{T}^n_{d})=1$.
We moreover remark that $R_{i,\tau}(\mathbb{T}^n_{d})$ can be interpreted as the ratio of the probabilities that the root gets color $i$ (resp. $q$) given the boundary condition $\tau$.

We define for $n \geq 0$, $\widehat{\mathbb{T}}^n_{d}$ to be the rooted tree obtained from $\mathbb{T}^n_{d}$ by adding a new root $\hat r_d$ connecting it to the original root $r_d$ with a single edge.
Note that the set of non-root leaves of $\widehat{\mathbb{T}}^n_{d} $ is just $\Lambda_{n,d}$.
For any boundary condition on $\tau:\Lambda_{n,d}\to [q]$ we define the restricted partition function, $Z_{i,\tau}(\widehat{\mathbb{T}}^n_{d})$ and ratio $R_{i,\tau}(\widehat{\mathbb{T}}^n_{d})$ analogously as for ${\mathbb{T}}^n_{d}$.

The next lemma provides a sufficient condition for ${\mathbb{T}}_{d}$ to have a unique Gibbs measure in terms of these ratios, which we prove at the end of this section. 

\begin{lemma}\label{lem:R=1impliesuniqueness}
   Let $q,d \in \mathbb{N}$ and $w\in(0,1)$.  Suppose that for all $i \in [q-1]$ and for all $\delta>0$ there exists $N >0$ such that for all $n \geq N$ and for all boundary conditions $\tau: \Lambda_{n,d} \to [q]$ we have
   \[
  |R_{i,\tau}(\widehat{\mathbb{T}}^n_{d})-1| < \delta,
   \]
   then the tree ${\mathbb{T}}_{d}$ has a unique Gibbs measure. 
  \end{lemma}

An advantage of working with the ratios of probabilities is that the well known tree recursion for the Potts model takes a convenient form.
\begin{lemma}\label{lem:treerecursion}
Let $n,d\in \mathbb{N}$ and let $\tau:\Lambda_{n,d}\to [q]$ be a boundary condition.
Let for $i=1,\ldots,d$, $T_i=\widehat{\mathbb{T}}^{n-1}_{d}$ be the components of $\mathbb{T}^n_d-r_d$ where we attach a new root vertex to $r_{d-1}$.
Let $\tau_i$ be the restriction of $\tau$ to $\Lambda_{n-1,d}\to [q]$ viewed as a subset of the vertices of $T_i$.
    Then we have for each $i \in [q-1],$
    \begin{equation}\label{eq:treeformula}
    R_{i,\tau}(\widehat{\mathbb{T}}^{n}_{d}) =  \frac{1+w \prod_{s=1}^{d} R_{i,\tau_s}(\widehat{\mathbb{T}}^{n-1}_{d})+\sum_{l \in [q-1]  \setminus \{i\}} \prod_{s=1}^{d} R_{l,\tau_s}(\widehat{\mathbb{T}}^{n-1}_{d}) }{w+\sum_{l \in [q-1] } \prod_{s=1}^{d} R_{l,\tau_s}(\widehat{\mathbb{T}}^{n-1}_{d})}.
\end{equation}
\end{lemma}
\noindent For completeness we provide a proof for this lemma at the end of this section.

A direct analysis of the recursion in Lemma~\ref{lem:treerecursion} is not straightforward, as it does not contract uniformly on a symmetric domain.
In~\cite{PottsLeslie3} this is remedied by looking at the two-step recursion, that is they analyze the behaviour of the ratio at depth $n$ as a function of the ratios at depth $n+2$.
They show with substantial, yet rigorous, aid of a computer algebra package that this two-step recursion does contract on a symmetric domain (when $q=3$ and $w$ and $d$ are as they should be).
We however take a different, more geometric approach and work instead with the one-step recursion, as described in the next subsection.
\end{subsection}

\begin{subsection}{A geometric condition for uniqueness}\label{sec: Intro functions}
To state a geometric condition, we first introduce some functions that allow us to treat the tree recursion from Lemma~\ref{lem:treerecursion} more concisely. 
Let $q \in \mathbb{Z}_{\geq 2},d \in \mathbb{Z}_{\geq 1}$ and $w \in [0,1)$. For $i \in [q]$ let $\mu_i$ be the map from $\mathbb{R}_{>0}^{q}$ to $\mathbb{R}_{>0}$ given by 
\[
    \mu_{i}(x_1,\dots,x_q) = (w -1)x_i + \sum_{\substack{j=1}}^{q} x_j.
\]
Furthermore, we define
\[
    \tilde{G}(x_1,\dots, x_q) = \left(\mu_1(x_1,\dots,x_{q}),\dots,\mu_{q}(x_1,\dots,x_{q}) \right)
\]
and
\[
    \tilde{F}(x_1,\dots, x_q) = \tilde{G}(x_1^d,\dots,x_{q}^d).
\]

Both $\tilde{F}$ and $\tilde{G}$ are homogeneous maps from $\mathbb{R}^{q}_{>0}$ to itself.
For $x,y \in \mathbb{R}_{>0}^{q}$ we define an equivalence relation $x\sim y$ if and only if $x = \lambda y$ for some $\lambda >0$. We define $\mathbb{P}^{q-1}_{>0}= \mathbb{R}_{>0}^{q} / \sim$ and denote elements of $\mathbb{P}^{q-1}_{>0}$ as $[x_1: \dots: x_q]$. We note that since $\tilde{F}$ and $\tilde{G}$ are homogeneous they are also well defined as maps from $\mathbb{P}^{q-1}_{>0}$ to itself and from now on we consider them as such.

Let $\pi: \mathbb{P}_{> 0}^{q-1} \to \mathbb{R}_{> 0}^{q-1}$ be the projection map defined by $\pi([x_1:\dots:x_q]) = (x_1/x_{q}, \dots, x_{q-1}/x_{q})$ with inverse $\iota: \mathbb{R}_{> 0}^{q-1} \to \mathbb{P}_{> 0}^{q-1}$ defined by $\iota(x_1, \dots, x_{q-1})=[x_1: \dots: x_{q-1}:1]$. Note that $\pi$ and $\iota$ are continuous. We define the maps $G,F$ from $\mathbb{R}_{> 0}^{q-1}$ to itself by $\pi \circ \tilde{G} \circ \iota$ and $\pi \circ \tilde{F} \circ \iota$.  Explicitly
    we have
    \[
        G(x_1,x_2) = \left(\frac{w x_1 + x_2 + 1}{x_1+x_2+w},\frac{x_1 + w x_2 + 1}{x_1+x_2+w}\right)\ \text{ and } \ F(x_1,x_2) = G(x_1^d,x_2^d) 
    \]
   for $q=3$. For $q=4$ we have 
    \[
        G(x_1,x_2,x_3) = \left(\frac{w x_1 + x_2 + x_3 + 1}{x_1+x_2 + x_3 +w},\frac{x_1 + w x_2 + x_3 +  1}{x_1+x_2 + x_3 +w},\frac{x_1 + x_2 + w x_3 +  1}{x_1+x_2 + x_3 +w}\right)
    \]
    and $F(x_1,x_2,x_3) = G(x_1^d,x_2^d,x_3^d)$.

With this definition the recursion from Lemma~\ref{lem:treerecursion} can now be stated as follows. 
Following the notation of the lemma, denote by $x_l$ the following log convex combination of the ratios $R_{l,\tau_s}(\widehat{\mathbb{T}}^{n-1}_{d})$,
\begin{equation}\label{eq:log convex combination}
x_l=\left(\prod_{s=1}^d R_{l,\tau_s}(\widehat{\mathbb{T}}^{n-1}_{d})\right)^{1/d}.
\end{equation}
Then
\begin{equation}\label{eq:recursion F}
\left(R_{1,\tau}(\widehat{\mathbb{T}}^{n}_{d}),\ldots,R_{q-1,\tau}(\widehat{\mathbb{T}}^{n}_{d})\right)=F(x_1,\ldots,x_{q-1}).
\end{equation}

We say that a subset $\mathcal{T} \subseteq \mathbb{R}_{>0}^n$ is \emph{log convex} if $\log\left(\mathcal{T}\right)$ is a convex subset of $\mathbb{R}^n$, where $\log\left(\mathcal{T}\right)$ denotes the set consisting of elements of $\mathcal{T}$ with the logarithm applied to their individual entries. 
The next lemma gives sufficient conditions for uniqueness on the infinite regular tree.

\begin{lemma}
    \label{lem: Tab sequence}
    Suppose that $q\geq 2$, $d\geq 2$ and $w>0$ are such that there exists a sequence $\{\mathcal{T}_n\}_{n\geq 0}$ of log convex subsets of $\mathbb{R}^{q-1}_{>0}$ with the following properties.
    \begin{enumerate}
        \item \label{it:basecase}
        Both the vector with every entry equal to $1/w$ and the vectors obtained from the all-ones vector with a single entry changed to $w$ are elements of $\mathcal{T}_0$.
        \item \label{it:inductionstep}
        For every $m$ we have $F(\mathcal{T}_m) \subseteq \mathcal{T}_{m+1}$.
        \item \label{it:laststep}
        For every $\epsilon > 0$ there is an $M$ such that for all $m\geq M$ every element of $\mathcal{T}_m$ has at most distance $\epsilon$ to the all-ones vector.
    \end{enumerate}
Then the anti-ferromagnetic Potts model with parameter $w$ has has a unique Gibbs measure on $\mathbb{T}_{d}$.
\end{lemma}
\begin{proof}
By Lemma~\ref{lem:R=1impliesuniqueness}, it suffices to show that regardless of the boundary condition $\tau$ on $\Lambda_{n,d}$,  $R_{i,\tau}(\widehat{\mathbb{T}}_d^n)\to 1$ as $n\to \infty$.

First of all we claim that for all $n \geq 0$ and all boundary conditions $\tau: \Lambda_{n,d} \to [q]$ we have $(R_{1,\tau}(\widehat{\mathbb{T}}^{n}_{d}),\ldots,R_{q-1,\tau}(\widehat{\mathbb{T}}^{n}_{d}))  \in \mathcal{T}_n$.
We prove this by induction on $n$. For the base case, $n=0$, we note that $\widehat{\mathbb{T}}^0_{d} $ consists of one free root $\hat r_{d}$, connected to a colored vertex $v$. 
If $v$ is colored $i \in [q-1]$, then $(R_{1,\tau}(\widehat{\mathbb{T}}^{0}_{d}),\ldots,R_{q-1,\tau}(\widehat{\mathbb{T}}^{0}_{d}))$ consists of a $w$ on position $i$ and ones everywhere else. 
If $v$ is colored $q$, then $(R_{1,\tau}(\widehat{\mathbb{T}}^{0}_{d}),\ldots,R_{q-1,\tau}(\widehat{\mathbb{T}}^{0}_{d})) = (1/w, \dots, 1/w)$.
So the base case follows from item (\ref{it:basecase}). 

Suppose next that for some $n \geq 0$ the claim holds. 
Let $\tau: \Lambda_{d,n+1} \to [q]$ be any boundary condition.
It then follows from \eqref{eq:log convex combination}, \eqref{eq:recursion F} and the assumptions that $\mathcal{T}_n$ is log convex and $F(\mathcal{T}_n)\subseteq (\mathcal{T}_{n+1})$ that
\[
\left(R_{1,\tau}(\widehat{\mathbb{T}}^{n}_{d}),\ldots,R_{q-1,\tau}(\widehat{\mathbb{T}}^{n}_{d})\right) \in \mathcal{T}_{n+1},
\]
completing the induction.


From the claim we just proved and item (\ref{it:laststep}) it then follows that given $\varepsilon>0$ there exists $N>0$ such that for all $n\geq N$, any boundary condition $\tau$ and color $i$, $|R_{i,\tau}(\widehat{\mathbb{T}}_d^n)- 1|<\varepsilon$.
This concludes the proof. 
\end{proof}

In the next section we will construct a sequence of regions $\{\mathcal{T}_n\}_{n \geq 0}$ satisfying the conditions of the lemma. In Subsection~\ref{sec: symmetry of F} we describe a certain symmetry that the map $F$ exhibits, corresponding to the symmetry of the colors in the Potts model. When a region $\mathcal{T}$ has a corresponding symmetry it is easier to understand the image $F(\mathcal{T})$. 
This is explained in Lemma~\ref{lem: One region is enough}. 
In Subsection~\ref{sec: regions Tab} we define a two parameter family of sets $\mathcal{T}_{a,b}$ that display the required symmetry. In Lemma~\ref{lem:convexity} we prove that if simple analytic conditions in $a$ and $b$ are satisfied the sets $\mathcal{T}_{a,b}$ are log-convex. 
In Lemma~\ref{lem: convex hulls} we give inner and outer approximations of the sets $\mathcal{T}_{a,b}$ with simple polytopes. This is convenient since the map $G$ is a fractional linear transformation and therefore preserves convex sets.
These are used in Lemma~\ref{lem:onecondition} in Subsection~\ref{sec: main thm proof} where we show that if more involved analytical conditions are satisfied $\mathcal{T}_{a,b}$ gets mapped strictly inside itself by $F$.
We then combine all ingredients to prove the \hyperref[thm:main]{Main Theorem} as Theorem~\ref{thm:precisemaintheorem}.  
In the proof of Theorem~\ref{thm:precisemaintheorem} we show that we can construct a sequence $\mathcal{T}_n = \mathcal{T}_{a_n,b_n}$ that satisfies the conditions of Lemma~\ref{lem: Tab sequence} using the fact that we can keep satisfying the analytic conditions on $a_n$ and $b_n$. 
This uses a number of technical inequalities whose verification we have moved to Section~\ref{sec: Inequality section} to preserve the flow of the text.

We finish this section be providing proofs of Lemma~\ref{lem:R=1impliesuniqueness} and Lemma~\ref{lem:treerecursion}.

\end{subsection}

\begin{subsection}{Proofs of Lemma~\ref{lem:R=1impliesuniqueness} and Lemma~\ref{lem:treerecursion}}
\begin{proof}[Proof of Lemma~\ref{lem:R=1impliesuniqueness}]
The ratios for $\widehat{\mathbb{T}}^n_{d}$ and $\mathbb{T}^n_{d}$ can easily be expressed in terms of each other. 
Fix any $\tau:\Lambda_{n,d}\to [q]$. Then for any $i=1,\ldots,q-1$,
\begin{equation}\label{eq:expressing R and hatR}
R_{i,\tau}(\widehat{\mathbb{T}}_d^n)=\frac{(w-1) R_{i,\tau}(\mathbb{T}_d^n)+\sum_{j=1}^{q-1} R_{j,\tau}(\mathbb{T}_d^n)+1}{\sum_{j=1}^{q-1}R_{j,\tau}(\mathbb{T}_d^n)+w}\quad \text{ and } \quad 
R_{i,\tau}(\mathbb{T}_d^n)=(R_{i,\tau}(\widehat{\mathbb{T}}_d^{n-1}))^d.
\end{equation}
We may thus assume that for each $\delta>0$ there exist $N'$ such that for all $n\geq N'$,$|R_{i,\tau}(\mathbb{T}_d^n)-1|<\delta$ for all $\tau:\Lambda_{n,d}\to [q].$

For readability, we omit the reference to the subscript $d$ in what follows.
For any $i=1,\ldots,q$ we have
\[
\PPr_{{\mathbb{T}}^{n}} [ {\bf\Phi}(r) = i \ \vert\ {\bf\Phi }\!\restriction_{\Lambda_n} = \tau] = \frac{Z_{i,\tau}({\mathbb{T}}^{n})}{Z_{\tau}{(\mathbb{T}}^{n})}
=
\frac{Z_{i,\tau}({\mathbb{T}}^{n})}{\sum_{j=1}^q Z_{j,\tau}({\mathbb{T}}^{n})}.
\]
Hence for any $i \in [q]$, upon dividing both the numerator and denominator by by $Z_{q,\tau}({\mathbb{T}}^{n})$, we obtain
\[
\PPr_{{\mathbb{T}}^{n}} [ {\bf \Phi}(r) = i \ \vert\ {\bf \Phi}\!\restriction_{ \Lambda_n} = \tau]  = \frac{R_{i,\tau}(\mathbb{T}^{n})}{\sum_{i=1}^{q-1} R_{i,\tau}(\mathbb{T}^{n}) +1}.
\]
Now since the map 
\[
(x_1, \dots, x_{q-1}) \mapsto \max_{i \in [q-1]} \bigg( \bigg \vert \frac{x_i}{\sum_{j=1}^{q-1} x_j +1} - \frac{1}{q}\bigg\vert, \bigg \vert \frac{1}{\sum_{j=1}^{q-1} x_j +1}  - \frac{1}{q}\bigg\vert\bigg)
\]
is continuous and maps $(1, \dots, 1)$ to $0$, it follows that for every $\epsilon>0$ there is a $\delta>0$ such that $|R_{i,\tau}(\mathbb{T}_{n})-1| < \delta$ for all boundary conditions $\tau$ and $i=1.\ldots,q-1$, implies that
\[
\max_{\tau: \Lambda_{n} \to [q]} \bigg\vert \PPr_{{\mathbb{T}}^{n}} [ {\bf \Phi}(r) = i \ \vert\ {\bf \Phi}\!\restriction_{\Lambda_n} = \tau] - \frac{1}{q}\bigg\vert < \epsilon .
\]
We conclude that the conditions of Lemma~\ref{def:uniqueness} are satisfied and hence $\mathbb{T}_d$ has a unique Gibbs measure.
\end{proof}

We next provide a proof for the tree recursion.

\begin{proof}[Proof of Lemma~\ref{lem:treerecursion}]
For readability we omit $d$ from the notation.
We have
\begin{equation}\label{eq:expand}
 R_{i,\tau}(\widehat{\mathbb{T}}^{n}) = \frac{Z_{i,\tau}(\widehat{\mathbb{T}}^{n})}{Z_{q,\tau}(\widehat{\mathbb{T}}^{n})} = \frac{\sum_{l \in [q]  \setminus \{i\}} Z_{l,\tau}({\mathbb{T}}^{n}) + w Z_{i,\tau}({\mathbb{T}}^{n})}{\sum_{l \in [q-1]  } Z_{l,\tau}({\mathbb{T}}^{n}) + w Z_{q,\tau}({\mathbb{T}}^{n})},
\end{equation}
as a factor $w$ is picked up when the unique neighbour of the root vertex is assigned the same color as the root vertex, $\hat{r}_d$, of $\widehat{\mathbb{T}}^{n}$. 
Note that for any color $c \in [q]$ we have $Z_{c,\tau}({\mathbb{T}}^{n}) = \prod_{s=1}^{d} Z_{c,\tau_s}(\widehat{\mathbb{T}}^{n-1})$.
Plugging this in into~\eqref{eq:expand} and dividing the numerator and denominator by $\prod_{s=1}^{d} Z_{q,\tau_s}(\widehat{\mathbb{T}}^{n-1})$, we arrive at the desired expression.
\end{proof}

\end{subsection}
\end{section}

\begin{section}{Proof of the main theorem}\label{sec:proof main}
In this section we use Lemma~\ref{lem: Tab sequence} to prove the \hyperref[thm:main]{Main Theorem}.

\begin{subsection}{Symmetry of the map $F$}
\label{sec: symmetry of F}
In order to find suitable sets $\mathcal{T}$ such that $F(\mathcal{T}) \subseteq \mathcal{T}$ we will exploit a symmetry that the map $F$ exhibits, due to the inherent symmetry of permuting the colors $[q]$ in the Potts model. To make this formal we will define a few self-maps on, and regions of, the spaces $\mathbb{P}^{q-1}_{>0}, \mathbb{R}^{q-1}_{>0}$ and $\mathbb{R}^{q-1}$.
To avoid confusion, we will denote self-maps on and subsets of $\mathbb{P}^{q-1}_{>0}$ with a tilde, self-maps on and subsets of $\mathbb{R}^{q-1}_{>0}$ without additional notation and self-maps on and subsets of $\mathbb{R}^{q-1}$ with a hat. When a self-map or subset is used as an index, we will drop the hat or tilde in the index. 

The three spaces $\mathbb{P}^{q-1}_{>0}, \mathbb{R}^{q-1}_{>0}$ and $\mathbb{R}^{q-1}$ are homeomorphic, with homeomorphisms $\pi: \mathbb{P}_{>0}^{q-1} \to \mathbb{R}_{>0}^{q-1}$ with inverse $\iota$ and $\log: \mathbb{R}_{>0}^{q-1} \to \mathbb{R}^{q-1}$ with inverse $\exp$. We define the self-maps $\hat{G},\hat{F}$ on $\mathbb{R}^{q-1}$ by $\hat{G} = \log \circ\ {G} \circ \exp$ and $\hat{F} = \log \circ\ {F} \circ \exp$. To summarize, we have the following diagram of continuous maps
\begin{equation*}
    \xymatrix{
    \mathbb{P}^{q-1}_{>0} \ar@<.5ex>[r]^-{\pi} \ar[d]^-{\tilde{G}, \tilde{F}}  & \mathbb{R}^{q-1}_{>0} \ar@<.5ex>[r]^-{\log} \ar[d]^-{G, F} \ar@<.5ex>[l]^-{\iota}& \mathbb{R}^{q-1} \ar@<.5ex>[l]^-{\exp} \ar[d]^-{\hat{G}, \hat{F}} 
    \\
    \mathbb{P}^{q-1}_{>0} \ar@<.5ex>[r]^-{\pi} &\mathbb{R}^{q-1}_{>0} \ar@<.5ex>[l]^-{\iota} \ar@<.5ex>[r]^-{\log} & \mathbb{R}^{q-1} \ar@<.5ex>[l]^-{\exp}\,.
    }
\end{equation*}
    
Let $S_q$ denote the symmetric group on $q$ elements. This group acts on $\mathbb{P}^{q-1}_{>0}$ be permuting the entries, which corresponds to permuting the colors in the Potts model. For $\sigma\in S_q$ we denote the map from $\mathbb{P}^{q-1}_{>0}$ to itself corresponding to this action by $\tilde{M}_\sigma$.
We use this action to also define an action on $\mathbb{R}_{>0}^{q-1}$ by letting $M_\sigma(x) = (\pi \circ \tilde{M}_\sigma \circ \iota)(x)$. It is easy to see that the action of $S_q$ on $\mathbb{P}_{>0}^{q-1}$ commutes with $\tilde{G}$ and $\tilde{F}$. It follows that the action of $S_q$ on $\mathbb{R}_{>0}^{q-1}$ also commutes with $F$ and $G$. 
Similarly, we define the map $\hat{M}_\sigma$ on $x \in \mathbb{R}^{q-1}$ by $\hat{M}_\sigma(x) = \left(\log \circ\ M_\sigma \circ \exp\right)(x)$ and we note that this action commutes with $\hat{G}$ and $\hat{F}$.
\begin{example}As an example we present the table of the action of $S_q$ for $q=3$ on a point in all the three coordinates.
\begin{equation*}
\begin{array}{c|cccccc}
    \sigma & \textrm{id} & (12) & (13) & (23) & (123) & (132) \\
    \hline 
    \tilde{M}_\sigma([x:y:z])& [x:y:z] & [y:x:z] &[z:y:x] & [x:z:y] & [z:x:y] & [y:z:x] \\
    M_\sigma(x,y) & (x,y) & (y,x) & (1/x,y/x) & (x/y,1/y) & (1/y,x/y) & (y/x,1/x) \\
    \hat{M}_\sigma(x,y) & (x,y) & (y,x) & (-x,y-x) & (x-y,-y) & (-y,x-y) & (y-x,-x) 
\end{array}
\end{equation*}
\end{example}
\noindent Note that in general $\hat{M}_\sigma$ is a linear map for all $\sigma \in S_q$. In fact, the map $\sigma \mapsto \hat{M}_\sigma$ is an irreducible representation of $S_q$ called the standard representation, but we will not use this.

    
For any permutation $\tau \in S_q$ we define the following subset of $\mathbb{P}_{>0}^{q-1}$
\[
    \tilde{\mathcal{R}}_{\tau} = \left\{[x_1: \dots: x_q] \in \mathbb{P}_{>0}^{q-1}:x_{\tau (1)} \leq x_{\tau (2)} \leq \cdots \leq x_{\tau(k)} \right\}.
\]
Furthermore, we let $\mathcal{R}_\tau = \pi(\tilde{\mathcal{R}}_{\tau})$ and $\hat{\mathcal{R}}_\tau = \log(\mathcal{R}_\tau)$.
Note that if $x\in\mathbb{P}_{>0}^{q-1}$ has the property that $x_i \geq x_j$ then $\mu_{i}(x) \leq \mu_{j}(x)$, recalling that 
\[
    \mu_{i}(x_1,\dots,x_q) = (w -1)x_i + \sum_{\substack{j=1}}^{q} x_j.
\]
It follows that the map $\tilde{G}$ maps $\tilde{\mathcal{R}}_\tau$ into $\tilde{\mathcal{R}}_{\tau \circ m}$, where $m \in S_q$ denotes the permutation with $m(l) = k+1-l$ for $l \in [q]$. The same is true for $\tilde{F}$ because $x \mapsto x^d$ maps any $\tilde{\mathcal{R}}_\tau$ to itself. It follows that $G$ and $F$ map $\mathcal{R}_\tau$ into $\mathcal{R}_{\tau \circ m}$ and that $\hat{G}$ and $\hat{F}$ map $\hat{\mathcal{R}}_\tau$ into $\hat{\mathcal{R}}_{\tau \circ m}$. In Figure \ref{fig:Regions} the regions $\mathcal{R}_\tau$ and $\hat{\mathcal{R}}_\tau$ are depicted when $q =3$.
    
The main purpose of the considerations of this section up until this point is to state and prove the following simple lemma.
\begin{lemma}
    \label{lem: One region is enough}
    Suppose $\mathcal{T} \subseteq \mathbb{R}_{>0}^{q-1}$ is a set such that $M_\sigma(\mathcal{T}) = \mathcal{T}$ for all $\sigma \in S_q$. Suppose also that there is a permutation $\tau\in S_q$ such that
    \[
        F\left(\mathcal{T} \cap \mathcal{R}_{\tau} \right) \subseteq \Int(\mathcal{T}).
    \]
    Then $F(\mathcal{T}) \subseteq \Int(\mathcal{T})$.
\end{lemma}
\begin{proof}
    Let $x \in \mathcal{T}$. There is a $\sigma \in S_q$ such that $M_\sigma(x) \in \mathcal{R}_{\tau}$ and thus $M_\sigma(x) \in \mathcal{T} \cap \mathcal{R}_{\tau}$. It follows from the assumption that $(F\circ M_\sigma)(x) \in \Int(\mathcal{T})$. Because $M_\sigma$ commutes with $F$ we find that $(M_\sigma \circ F)(x) \in \Int(\mathcal{T})$. We conclude that $F(x) \in M_\sigma^{-1}(\Int(\mathcal{T}))=M_{\sigma^{-1}}(\Int(\mathcal{T}))\subseteq \mathcal{T}$. Because $M_\sigma$ is continuous it follows that $M_\sigma^{-1}(\Int(\mathcal{T}))$ is an open subset of $\mathcal{T}$ and hence $F(x) \in \Int(\mathcal{T})$.
\end{proof}
In the next section we will define a family of regions $\mathcal{T}_{a,b}$ for $a,b>1$ with the property $M_\sigma(\mathcal{T}_{a,b})=\mathcal{T}_{a,b}$ for all $\sigma \in S_q$. Our goal will be to show that for certain choices of parameters $(a,b)$ we have $F(\mathcal{T}_{a,b}) \subseteq \Int(T_{a,b})$. Because of Lemma~\ref{lem: One region is enough} it will be enough to restrict ourselves to one well chosen region $\mathcal{R}_{\tau}$. 
\end{subsection}

\begin{subsection}{Definition and properties of the sets $\mathcal{T}_{a,b}$}
\label{sec: regions Tab}

For $q=3$ and $q=4$ we will define a family of log convex sets $\mathcal{T}_{a,b}\subseteq \mathbb{R}_{>0}^{q-1}$ with the property that $M_\sigma(\mathcal{T}_{a,b}) = \mathcal{T}_{a,b}$ for all $\sigma \in S_q$. We will do this by defining the convex sets $\hat{\mathcal{T}}_{a,b} \subseteq \mathbb{R}^{q-1}$ and then letting $\mathcal{T}_{a,b} = \exp(\hat{\mathcal{T}}_{a,b})$.

Let $a,b>1$. To avoid having to write too many logarithms we let $\hat{a} = \log(a)$ and $\hat{b} = \log(b)$. For $q=3$ we define the following half-space of $\mathbb{R}^2$
\[
    \hat{H}_{a,b} = \{(x,y) \in \mathbb{R}^2: -\hat{b}\cdot x + \hat{a}\cdot y \leq \hat{a}\hat{b}\}.
\]
Subsequently, we define 
\[
    \hat{\mathcal{T}}_{a,b} = \bigcup_{\sigma \in S_3}\hat{M}_\sigma\left(\hat{\mathcal{R}}_{(23)} \cap \hat{H}_{a,b}\right).
\]
Similarly, for $q=4$, we define the half-space
\[
    \hat{H}_{a,b} = \{(x,y,z) \in \mathbb{R}^3: -\hat{b}\cdot x + \hat{a}\cdot z \leq \hat{a}\hat{b}\}
\]
and the region 
\[
    \hat{\mathcal{T}}_{a,b} = \bigcup_{\sigma \in S_4}\hat{M}_\sigma\left(\hat{\mathcal{R}}_{(243)} \cap \hat{H}_{a,b}\right).
\]
For both $q=3$ and $q=4$ we let $\mathcal{T}_{a,b} = \exp(\hat{\mathcal{T}}_{a,b})$. Figure \ref{fig:Regions} contains an image of $\hat{\mathcal{T}}_{a,b}$ and $\mathcal{T}_{a,b}$ for $q=3$. Figure \ref{fig:Region3D} contains an image of $\hat{\mathcal{T}}_{a,b}$ for $q=4$; we highlighted the region $\hat{\mathcal{R}}_{(243)} \cap \hat{H}_{a,b}$ in orange. We have chosen to give the sets $\hat{\mathcal{T}}_{a,b}$ the same name for $q=3$ and $q=4$. This is because many of the properties of $\hat{\mathcal{T}}_{a,b}$ that we will prove hold for both cases and are proved in a similar way. Unless otherwise stated one should assume that any statement involving $\hat{\mathcal{T}}_{a,b}$ refers to the corresponding statement for both $q=3$ and $q=4$. 

We first state a basic lemma relating the half-space representation and the vertex representation of a polytope. This lemma will be used a number of times in the remainder of the section to derive useful properties of the sets $\mathcal{T}_{a,b}$.

\begin{lemma}
    \label{lem: half-spaces}
    Let $H_1, \dots, H_n$ be closed half-spaces in $\mathbb{R}^{n-1}$. 
    Furthermore, let $p_1, \dots, p_n \in \mathbb{R}^{n-1}$ with the property that for all $i \in [n]$ we have $p_i \in \Int(H_i)$ and $p_i \in \partial H_j$ for $j \neq i$.
    Then 
    \[
        \bigcap_{i=1}^n H_i = \Conv\left(\{p_1,\dots, p_n\}\right),
    \]
    where $\Conv(S)$ denotes the convex hull of the set $S$.
\end{lemma}
\begin{proof}
We give a sketch of the proof. The conditions on the $p_i$ imply that the set $\{p_1, \dots, p_n\}$ is affinely independent, i.e. the set $\{p_1-p_n, \dots, p_{n-1}-p_n\}$ is linearly independent. Therefore there exists an invertible affine transformation $T$ with $T(v) = M(v-p_n)$ for some invertible linear transformation $M$, such that $T(p_i) = e_i$ for $i \in [n-1]$, where $e_i$ denotes a standard basis vector. From the conditions on the $p_i$ it follows that $T(H_i) = \{x \in \mathbb{R}^{n-1}:  x_i \geq 0 \}$ for $i \in [n-1]$ and $T(H_n) = \{x \in \mathbb{R}^{n-1}: \sum_{i=1}^{n-1} x_i \leq 1\}$. As affine transformations preserve convexity, we see
\begin{align*}
    T\left(\bigcap_{i=1}^n H_i\right)
    &= 
    \bigcap_{i=1}^n T(H_i)
    =
    \bigcap_{i=1}^{n-1} \{x \in \mathbb{R}^{n-1}:  x_i \geq 0 \} \cap \{x \in \mathbb{R}^{n-1}: \sum_{i=1}^{n-1} x_i \leq 1\}\\
    &= \Conv\left(\{e_1,\dots, e_{n-1},0\}\right)
    = \Conv\left(\{T(p_1), \dots, T(p_n)\}\right)
    = T\left(\Conv\left(\{p_1,\dots, p_n\} \right)\right).
\end{align*}
The lemma now follows from the fact that $T$ is invertible.
\end{proof}


\begin{figure}[h!]
\centering
\begin{tikzpicture}
	\draw [draw=black, fill=orange, fill opacity=0.75]
	 (-3.29584, 0.) -- (0., 2.07944) -- (3.29584, 3.29584) -- (2.07944, 0.) -- (0., -3.29584) -- (-2.07944, -2.07944) -- cycle;
	 
	 \draw [dashed] (-3.6, 0.) -- (3.6, 0.);
	 \draw [dashed] (0., -3.6) -- (0., 3.6);
	 \draw [dashed] (-3.6, -3.6) -- (3.6, 3.6);
	 
	 \node [anchor = east,fill=white,inner sep=1pt] at (-3.29584, 0.) {$(-\hat{a},0)$};
	 \node [anchor = east] at (0., 2.07944) {$(0,\hat{b})$};
	 \node [anchor = south east] at (3.29584, 3.29584) {$(\hat{a},\hat{a})$};
	 \node [anchor = west,fill=white, inner sep=0.01pt] at (2.14544, 0.) {$(\hat{b},0)$};
	 \node [anchor = north,fill=white,inner sep=1pt] at (0., -3.29584) {$(0,-\hat{a})$};
	 \node [anchor = east,fill=white,inner sep=0.01pt] at (-2.23944, -2.07944) {$(-\hat{b},-\hat{b})$};
	 
	 \node[circle,draw=black, fill=blue, inner sep=0pt,minimum size=3pt] (a) at (-3.29584, 0.) {};
	 \node[circle,draw=black, fill=blue, inner sep=0pt,minimum size=3pt] (b) at (0., 2.07944) {};
	 \node[circle,draw=black, fill=blue, inner sep=0pt,minimum size=3pt] (c) at (3.29584, 3.29584) {};
	 \node[circle,draw=black, fill=blue, inner sep=0pt,minimum size=3pt] (d) at (2.07944, 0.) {};
	 \node[circle,draw=black, fill=blue, inner sep=0pt,minimum size=3pt] (e) at (0., -3.29584) {};
	 \node[circle,draw=black, fill=blue, inner sep=0pt,minimum size=3pt] (f) at (-2.07944, -2.07944) {};
	 
	 \node at (-2.4, 2.4) {$\hat{\mathcal{R}}_{(2,3)}$};
	 \node at (1.2, 3.) {$\hat{\mathcal{R}}_{(1,3,2)}$};
	 \node at (3., 0.6) {$\hat{\mathcal{R}}_{(1,3)}$};
	 \node at (1.5, -2.4) {$\hat{\mathcal{R}}_{(1,2,3)}$};
	 \node at (-1.8, -3.) {$\hat{\mathcal{R}}_{(1,2)}$};
	 \node at (-3.6, -1.2) {$\hat{\mathcal{R}}_{e}$};
	 
	 \node at (-1.2, 0.5) {$\hat{\mathcal{R}}_{(2,3)} \cap \hat{H}_{a,b}$};

	 
	 \draw [draw=black, fill=orange, fill opacity=0.75]
(4.35, -1.35) -- (4.3837, -1.28674) -- (4.4189, -1.22171) -- (4.45569, -1.15484) -- (4.49413, -1.0861) -- (4.5343, -1.01543) -- (4.57627, -0.942768) -- (4.62013, -0.868062) -- (4.66596, -0.791256) -- (4.71385, -0.71229) -- (4.76389, -0.63111) -- (4.81617, -0.54764) -- (4.87081, -0.461827) -- (4.9279, -0.373601) -- (4.98755, -0.282894) -- (5.04989, -0.189638) -- (5.11502, -0.0937596) -- (5.18308, 0.00481445) -- (5.2542, 0.10616) -- (5.32852, 0.210354) -- (5.40617, 0.317477) -- (5.48731, 0.427613) -- (5.5721, 0.540844) -- (5.66069, 0.657258) -- (5.75327, 0.776946) -- (5.85, 0.9) -- (5.95108, 0.973578) -- (6.0567, 1.04836) -- (6.16704, 1.12436) -- (6.28236, 1.20161) -- (6.4029, 1.28012) -- (6.52884, 1.35992) -- (6.66036, 1.44101) -- (6.79788, 1.52344) -- (6.94152, 1.60721) -- (7.09164, 1.69235) -- (7.24854, 1.77889) -- (7.4124, 1.86684) -- (7.5837, 1.95623) -- (7.76268, 2.04707) -- (7.94964, 2.13941) -- (8.14506, 2.23325) -- (8.34924, 2.32864) -- (8.5626, 2.42557) -- (8.78556, 2.52409) -- (9.01848, 2.62423) -- (9.2619, 2.726) -- (9.5163, 2.82944) -- (9.7821, 2.93456) -- (10.0598, 3.04141) -- (10.35, 3.15) -- (10.2414, 2.8598) -- (10.1345, 2.58207) -- (10.0294, 2.31629) -- (9.92598, 2.06193) -- (9.82422, 1.8185) -- (9.72408, 1.58555) -- (9.62556, 1.36261) -- (9.52866, 1.14925) -- (9.43326, 0.945066) -- (9.33942, 0.749658) -- (9.24708, 0.562655) -- (9.15624, 0.38369) -- (9.06684, 0.21242) -- (8.97888, 0.0485128) -- (8.89236, -0.108347) -- (8.80722, -0.258464) -- (8.72346, -0.402126) -- (8.64102, -0.539612) -- (8.5599, -0.67119) -- (8.4801, -0.797106) -- (8.40162, -0.91761) -- (8.32434, -1.03294) -- (8.24838, -1.1433) -- (8.17356, -1.24892) -- (8.1, -1.35) -- (7.97694, -1.44673) -- (7.85724, -1.53931) -- (7.74084, -1.6279) -- (7.62762, -1.71269) -- (7.51746, -1.79383) -- (7.41036, -1.87148) -- (7.30614, -1.9458) -- (7.2048, -2.01692) -- (7.10622, -2.08498) -- (7.01034, -2.15011) -- (6.9171, -2.21245) -- (6.82638, -2.2721) -- (6.73818, -2.32919) -- (6.65238, -2.38383) -- (6.56892, -2.43611) -- (6.48768, -2.48615) -- (6.40872, -2.53404) -- (6.33192, -2.57987) -- (6.25722, -2.62373) -- (6.18456, -2.6657) -- (6.11388, -2.70587) -- (6.04518, -2.74431) -- (5.97829, -2.7811) -- (5.91326, -2.8163) -- (5.85, -2.85) -- (5.78848, -2.83774) -- (5.72863, -2.82527) -- (5.67042, -2.81261) -- (5.6138, -2.79973) -- (5.55874, -2.78665) -- (5.50517, -2.77335) -- (5.45308, -2.75983) -- (5.40241, -2.74609) -- (5.35312, -2.73213) -- (5.30518, -2.71794) -- (5.25855, -2.70352) -- (5.2132, -2.68886) -- (5.16908, -2.67396) -- (5.12618, -2.65882) -- (5.08445, -2.64343) -- (5.04385, -2.62779) -- (5.00437, -2.6119) -- (4.96597, -2.59574) -- (4.92862, -2.57932) -- (4.89229, -2.56263) -- (4.85695, -2.54567) -- (4.82258, -2.52843) -- (4.78915, -2.5109) -- (4.75663, -2.4931) -- (4.725, -2.475) -- (4.7069, -2.44337) -- (4.6891, -2.41085) -- (4.67157, -2.37742) -- (4.65433, -2.34305) -- (4.63737, -2.30771) -- (4.62068, -2.27138) -- (4.60426, -2.23403) -- (4.5881, -2.19563) -- (4.57221, -2.15615) -- (4.55657, -2.11555) -- (4.54118, -2.07382) -- (4.52604, -2.03092) -- (4.51114, -1.9868) -- (4.49648, -1.94145) -- (4.48206, -1.89482) -- (4.46787, -1.84688) -- (4.45391, -1.79759) -- (4.44017, -1.74692) -- (4.42665, -1.69483) -- (4.41335, -1.64126) -- (4.40027, -1.5862) -- (4.38739, -1.52958) -- (4.37473, -1.47137) -- (4.36226, -1.41152) -- cycle;

\draw [dashed] (3.6, -1.35) -- (10.8, -1.35);
\draw [dashed] (5.85, -3.6) -- (5.85, 3.6);
\draw [dashed] (3.6, -3.6) -- (10.8, 3.6);

	 \node [anchor = east,fill=white,inner sep=1pt] at (4.35, -1.35) {$(1/a,1)$};
	 \node [anchor = east] at (5.85, 0.9) {$(1,b)$};
	 \node [anchor = south east] at (10.35, 3.15) {$(a,a)$};
	 \node [anchor = west,fill=white, inner sep=0.01pt] at (8.19, -1.35) {$(b,1)$};
	 \node [anchor = north,fill=white,inner sep=1pt] at (5.85, -2.85) {$(1,1/a)$};
	 \node [anchor = east,fill=white,inner sep=0.01pt] at (4.635, -2.475) {$(1/b,1/b)$};
	 
	 \node[circle,draw=black, fill=blue, inner sep=0pt,minimum size=3pt] at (4.35, -1.35) {};
	 \node[circle,draw=black, fill=blue, inner sep=0pt,minimum size=3pt] at (5.85, 0.9) {};
	 \node[circle,draw=black, fill=blue, inner sep=0pt,minimum size=3pt] at (10.35, 3.15) {};
	 \node[circle,draw=black, fill=blue, inner sep=0pt,minimum size=3pt] at (8.1, -1.35) {};
	 \node[circle,draw=black, fill=blue, inner sep=0pt,minimum size=3pt] at (5.85, -2.85) {};
	 \node[circle,draw=black, fill=blue, inner sep=0pt,minimum size=3pt] at (4.725, -2.475) {};
	 
	 \node at (4.725, 0.5) {$\mathcal{R}_{(2,3)}$};
	 \node at (6.975, 2.125) {$\mathcal{R}_{(1,3,2)}$};
	 \node at (9.425, -0.225) {$\mathcal{R}_{(1,3)}$};
	 \node at (8.1, -2.475) {$\mathcal{R}_{(1,2,3)}$};
	 \node at (4.775, -3.325) {$\mathcal{R}_{(1,2)}$};
	 \node at (4.275, -2.025) {$\mathcal{R}_{e}$};
	\end{tikzpicture}

\caption{Images for $q=3$ of $\hat{\mathcal{T}}_{a,b}$ on the left and $\mathcal{T}_{a,b}$ on the right. The boundaries of the regions $\hat{\mathcal{R}}_\tau$ and $\mathcal{R}_\tau$ are drawn with dashed lines.}
\label{fig:Regions}
\end{figure}
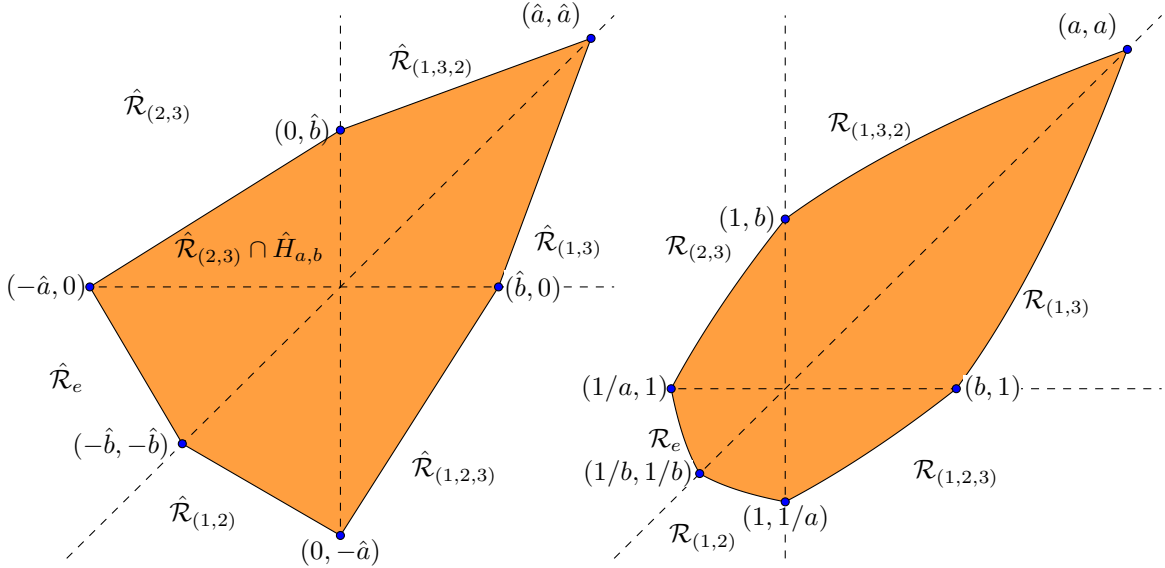

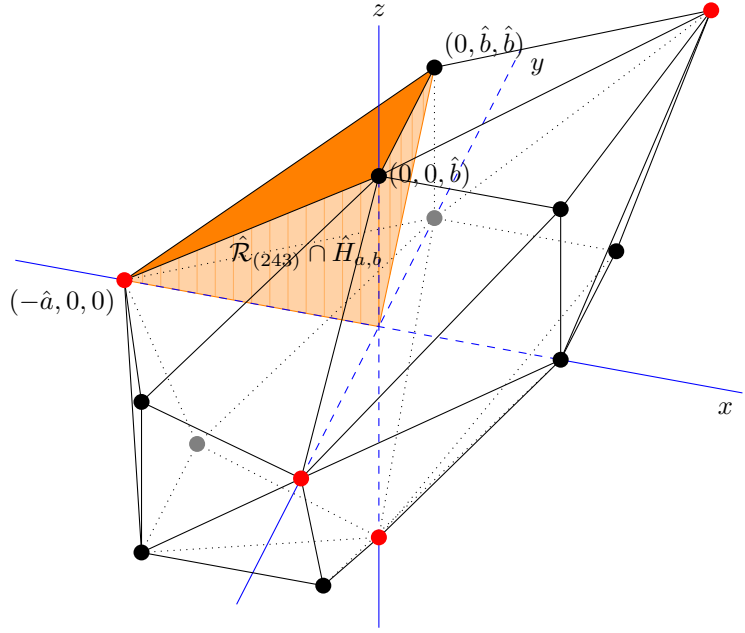
\begin{figure}[h!]
    \centering
    \tdplotsetmaincoords{53}{107}
\begin{tikzpicture}[tdplot_main_coords,font=\sffamily]
\fill[orange] (0., -3.5, 0.)-- (-2.5,0,2.5) -- (0,0,2.5);

\fill[orange, opacity=0.35] (0., 0, 0.)-- (-2.5,0,2.5) -- (0,0,2.5);

\fill[orange, opacity=0.35] (0., -3.5, 0.)-- (0,0,0) -- (0,0,2.5);

\draw[orange, opacity=0.4 ,-] (0., -3, 0.) -- (0., -3, 0.5);
\draw[orange, opacity=0.4 ,-] (0., -2.5, 0.) -- (0., -2.5, 1);
\draw[orange, opacity=0.4 ,-] (0., -2, 0.) -- (0., -2, 1.5);
\draw[orange, opacity=0.4 ,-] (0., -1.5, 0.) -- (0., -1.5, 2);
\draw[orange, opacity=0.4 ,-] (0., -1, 0.) -- (0., -1, 2.5);
\draw[orange, opacity=0.4 ,-] (0., -0.5, 0.) -- (0., -0.5, 2.5);

\draw[orange, opacity=0.4 ,-] (0., -3-0.25, 0.) -- (0., -3-0.25, 0.2);
\draw[orange, opacity=0.4 ,-] (0., -2.5-0.25, 0.) -- (0., -2.5-0.25, 0.6);
\draw[orange, opacity=0.4 ,-] (0., -2-0.25, 0.) -- (0., -2-0.25, 1.2);
\draw[orange, opacity=0.4 ,-] (0., -1.5-0.25, 0.) -- (0., -1.5-0.25, 1.6);
\draw[orange, opacity=0.4 ,-] (0., -1-0.25, 0.) -- (0., -1-0.25, 2.1);
\draw[orange, opacity=0.4 ,-] (0., -0.5-0.25, 0.) -- (0., -0.5-0.25, 2.5);
\draw[orange, opacity=0.4 ,-] (0., 0-0.25, 0.) -- (0., 0-0.25, 2.5);

\draw[orange, opacity=0.4 ,-] (-2, 0, 2) -- (-2,0,2.5);
\draw[orange, opacity=0.4 ,-] (-1.5, 0, 1.5) -- (-1.5,0,2.5);
\draw[orange, opacity=0.4 ,-] (-1, 0, 1) -- (-1,0,2.5);
\draw[orange, opacity=0.4 ,-] (-0.5, 0, 0.5) -- (-0.5,0,2.5);

\draw[orange,-] (0, 0, 0.) -- (-2.5,0,2.5);
\draw[orange,-] (0, 0, 0.) -- (0,-3.5,0);
\draw[orange,-] (0, 0, 0.) -- (0,0,2.5);

\draw[blue,-] (0,-5,0) -- (0,-3.5,0);
\draw[blue,dashed,-] (0,-3.5,0) -- (0,2.5,0);
\draw[blue,-] (0,2.5,0) -- (0,5,0) node[black,anchor=north east]{$x$};

\draw[blue,-] (0,0,-5) -- (0,0,-3.5);
\draw[blue,dashed,-] (0,0,-3.5) -- (0,0,2.5);

\draw[blue,dashed,-] (3.5,0,0) -- (-6.4,0,0) node[black,anchor=north west]{$y$};
\draw[blue,-] (3.5,0,0) -- (6.4,0,0);

\draw[black,-] (-2.5, 0., 2.5) -- (0., 0., 2.5);
\draw[black,-] (0., 2.5, 2.5) -- (0., 0., 2.5);
\draw[black,-] (-2.5, 2.5, 0.) -- (0., 2.5, 0.);
\draw[dotted,black,-] (-2.5, 2.5, 0.) -- (-2.5, 0., 0.);
\draw[black,-] (0., 2.5, 2.5) -- (0., 2.5, 0.);
\draw[black,-] (2.5, -2.5, 0.) -- (0., 0., 2.5);
\draw[dotted,black,-] (0., -2.5, -2.5) -- (2.5, -2.5, -2.5);
\draw[dotted,black,-] (-2.5, 0., 2.5) -- (-2.5, 0., 0.);
\draw[black,-] (2.5, 0., -2.5) -- (2.5, -2.5, -2.5);
\draw[dotted,black,-] (2.5, 0., -2.5) -- (0., 2.5, 0.);
\draw[dotted,black,-] (0., -2.5, -2.5) -- (-2.5, 0., 0.);
\draw[black,-] (2.5, -2.5, 0.) -- (2.5, -2.5, -2.5);
\draw[black,-] (2.5, 0., -2.5) -- (0., 0., -3.5);
\draw[dotted,black,-] (-2.5, 0., 0.) -- (0., 0., -3.5);
\draw[black,-] (0., 0., 2.5) -- (3.5, 0., 0.);
\draw[black,-] (2.5, 0., -2.5) -- (3.5, 0., 0.);
\draw[dotted,black,-] (0., -2.5, -2.5) -- (0., 0., -3.5);
\draw[dotted,black,-] (2.5, -2.5, -2.5) -- (0., 0., -3.5);
\draw[black,-] (2.5, -2.5, 0.) -- (3.5, 0., 0.);
\draw[black,-] (2.5, -2.5, -2.5) -- (3.5, 0., 0.);
\draw[black,-] (0., 2.5, 0.) -- (0., 0., -3.5);
\draw[dotted,black,-] (-2.5, 2.5, 0.) -- (0., 0., -3.5);
\draw[black,-] (0., 2.5, 0.) -- (3.5, 0., 0.);
\draw[black,-] (0., 2.5, 2.5) -- (3.5, 0., 0.);
\draw[black,-] (0., 0., 2.5) -- (0., -3.5, 0.);
\draw[dotted,black,-] (-2.5, 0., 0.) -- (0., -3.5, 0.);
\draw[black,-] (-2.5, 0., 2.5) -- (0., -3.5, 0.);
\draw[black,-] (0., 0., 2.5) -- (-3.5, 3.5, 3.5);
\draw[dotted,black,-] (-2.5, 0., 0.) -- (-3.5, 3.5, 3.5);
\draw[black,-] (-2.5, 0., 2.5) -- (-3.5, 3.5, 3.5);
\draw[dotted,black,-] (0., -2.5, -2.5) -- (0., -3.5, 0.);
\draw[black,-] (2.5, -2.5, 0.) -- (0., -3.5, 0.);
\draw[black,-] (2.5, -2.5, -2.5) -- (0., -3.5, 0.);
\draw[black,-] (0., 2.5, 0.) -- (-3.5, 3.5, 3.5);
\draw[black,-] (0., 2.5, 2.5) -- (-3.5, 3.5, 3.5);
\draw[black,-] (-2.5, 2.5, 0.) -- (-3.5, 3.5, 3.5);

\draw[blue,-] (0,0,2.5) -- (0,0,5) node[black,anchor=south]{$z$};

\fill[black] (0., 0., 2.5) circle (3pt);
\fill[black] (-2.5, 0., 2.5) circle (3pt);
\fill[black] (0., 2.5, 2.5) circle (3pt);
\fill[black] (0., 2.5, 0.) circle (3pt);
\fill[black] (-2.5, 2.5, 0.) circle (3pt);
\fill[gray] (-2.5, 0., 0.) circle (3pt);
\fill[black] (2.5, 0., -2.5) circle (3pt);
\fill[gray] (0., -2.5, -2.5) circle (3pt);

\fill[black] (2.5, -2.5, 0.) circle (3pt);
\fill[black] (2.5, -2.5, -2.5) circle (3pt);

\fill[red] (0., -3.5, 0.) circle (3pt);
\fill[red] (3.5, 0., 0.) circle (3pt);
\fill[red] (0., 0., -3.5) circle (3pt);
\fill[red] (-3.5, 3.5, 3.5) circle (3pt);

\node [below left] at (0, -3.5, 0.) {$(-\hat{a}, 0, 0)$};

\node [above right] at (-2.5,0,2.5) {$(0,\hat{b}, \hat{b})$};
\node [right] at (0,0,2.55) {$(0,0, \hat{b})$};

\node at (0, -1, 1) {$\hat{\mathcal{R}}_{(243)} \cap \hat{H}_{a,b}$};

\end{tikzpicture}
    \caption{Image for $q=4$ of $\hat{\mathcal{T}}_{a,b}$. The red dots depend on $\hat{a}$ while the black dots depend on $\hat{b}$. In orange the region $\hat{\mathcal{R}}_{(243)} \cap \hat{H}_{a,b}$ is depicted.}
    \label{fig:Region3D}
\end{figure}

\begin{lemma}\label{lem:convexity}
    For $a,b \in \mathbb{R}_{>1}$ with $b\leq a \leq b^2$ we have that $\hat{\mathcal{T}}_{a,b}$ is convex, or equivalently, that $\mathcal{T}_{a,b}$ is log convex.
\end{lemma}
\begin{proof}
    Recall that we let $\hat{a} = \log(a)$ and $\hat{b} = \log(b)$ and observe that these are two positive real numbers.
    Also recall that the action of $S_q$ on $\mathbb{R}^{q-1}$ is given by linear maps. It follows that the half-space $\hat{H}_{a,b}$ gets mapped to a half-space by $\hat{M}_\sigma$ for any $\sigma \in S_q$. We will show that for the choices of parameters stated in the lemma we have 
    \begin{equation}
        \label{eq: intersection half-spaces}
        \hat{\mathcal{T}}_{a,b} = \bigcap_{\sigma\in S_q}\hat{M}_\sigma\left(\hat{H}_{a,b}\right)
    \end{equation}
    for both $q=3$ and $q=4$. This equality implies that $\hat{\mathcal{T}}_{a,b}$ is convex because an intersection of half-spaces is convex. In fact, it implies that $\hat{\mathcal{T}}_{a,b}$ is a convex polytope. 
    
    We will first prove that the right-hand side of (\ref{eq: intersection half-spaces}) is contained in the left-hand side. To that effect take an element $x \in \bigcap_{\sigma\in S_q}\hat{M}_\sigma\left(\hat{H}_{a,b}\right)$. Because the collection $\{\hat{\mathcal{R}}_{\sigma}\}_{\sigma\in S_q}$ covers $\mathbb{R}^{q-1}$, there is a $\tau \in S_q$ such that $x \in \hat{\mathcal{R}}_{\tau}$. For $q = 3$ let $\sigma \in S_3$ such that $\sigma \cdot (23) = \tau$. We see that $x \in \hat{\mathcal{R}}_{\tau}\cap \hat{M}_\sigma(\hat{H}_{a,b})$ and thus $x \in \hat{M}_\sigma \left(\hat{\mathcal{R}}_{(23)}\cap \hat{H}_{a,b}\right)$, from which it follows $x \in \hat{\mathcal{T}}_{a,b}$. Similarly, for $q=4$ we let $\sigma \in S_4$ such that $\sigma \cdot (243) = \tau$. It follows in exactly the same way that $x \in \hat{\mathcal{T}}_{a,b}$.
    
    The proof that the left-hand side of (\ref{eq: intersection half-spaces}) is contained in the right-hand side is slightly more involved. Assume that $q=3$. We first show that
    \begin{equation}
    \label{eq: convhul1}
        \hat{\mathcal{R}}_{(23)} \cap \hat{H}_{a,b} = \Conv\left(\{(0,0),(-\hat{a},0),(0,\hat{b})\}\right),
    \end{equation}
    While this is easily seen to be true from Figure~\ref{fig:Regions}, we provide a formal proof. Note that $\hat{\mathcal{R}}_{(23)}$ is the intersection of $\hat{H}_{x\leq 0} = \{(x,y) \in \mathbb{R}^2: x \leq 0\}$ and $\hat{H}_{y \geq 0} = \{(x,y) \in \mathbb{R}^2: y \geq 0\}$. One can check that $(0,0) \in \partial \hat{H}_{x\leq 0} \cap \partial \hat{H}_{y\geq 0}\cap\Int(\hat{H}_{a,b})$,
        $(-\hat{a},0) \in \partial \hat{H}_{a,b} \cap \partial \hat{H}_{y\geq 0}\cap\Int(\hat{H}_{x\leq 0} )$
    and $(0,\hat{b}) \in \partial \hat{H}_{a,b} \cap \partial \hat{H}_{x\leq 0}\cap\Int(\hat{H}_{y\geq 0} )$. Equation~{(\ref{eq: convhul1})} then follows from Lemma~\ref{lem: half-spaces}.

    
    We obtain
    \[
        \hat{\mathcal{T}}_{a,b} 
        =
        \bigcup_{\sigma \in S_3}\hat{M}_\sigma\left(\Conv\left(\{(0,0),(-\hat{a},0),(0,\hat{b})\}\right)\right)
        =
        \bigcup_{\sigma \in S_3}\Conv\left(\{(0,0),\hat{M}_\sigma(-\hat{a},0),\hat{M}_\sigma(0,\hat{b})\}\right).
    \]
    We want to show that this is a subset of $\bigcap_{\sigma\in S_3}\hat{M}_\sigma\left(\hat{H}_{a,b}\right)$. Because all these half-spaces are convex, it is enough to show that the set $P = \{(0,0)\} \cup \bigcup_{\sigma \in S_3} \{\hat{M}_\sigma(-\hat{a},0),\hat{M}_\sigma(0,\hat{b})\}$ is a subset of $\hat{M}_\tau(H_{a,b})$ for all $\tau \in S_3$. Because the set $P$ is invariant under the action of $S_3$ it is sufficient to show that $P \subseteq \hat{H}_{a,b}$. We can calculate $P$ explicitly to obtain
    \[
        P = \{(0,0),(0, -\hat{a}), (-\hat{a}, 0), (\hat{a}, \hat{a}),(0, \hat{b}), (\hat{b}, 0),(-\hat{b}, -\hat{b})\}.
    \]
    To check that these points lie in $\hat{H}_{a,b}$ we have to check that for each $(x,y) \in P$ we have $-\hat{b}\cdot x + \hat{a}\cdot y \leq \hat{a}\hat{b}$. The inequality is trivially true for all but the points $(\hat{a}, \hat{a})$ and $(-\hat{b}, -\hat{b})$. One can confirm that the inequalities obtained by filling in these two points are simultaneously satisfied if and only if $\hat{b}/2 \leq \hat{a} \leq 2\hat{b}$. Because $\hat{a} = \log(a)$ and $\hat{b}= \log(b)$ this is equivalent to $\sqrt{b} \leq a \leq b^2$. This shows that for these choices of $a$ and $b$ the left-hand side of (\ref{eq: intersection half-spaces}) is contained in the right-hand side, which concludes the proof for $q=3$.
    
    The proof for $q=4$ follows the same path. One can show in very similar way to the $q=3$ case that
    \[
        \hat{\mathcal{R}}_{(243)} \cap \hat{H}_{a,b} = \Conv\left(\{(0,0,0),(-\hat{a},0,0),(0,0,\hat{b}),(0,\hat{b},\hat{b})\}\right)
    \]
    and thus that 
    \[
        \hat{\mathcal{T}}_{a,b} 
        =
        \bigcup_{\sigma \in S_4}\Conv\left(\{(0,0,0),\hat{M}_\sigma(-\hat{a},0,0),\hat{M}_\sigma(0,0,\hat{b}),\hat{M}_\sigma(0,\hat{b},\hat{b})\}\right).
    \]
    It is again sufficient to show that $P = \{(0,0,0)\} \cup \bigcup_{\sigma \in S_4} \{\hat{M}_\sigma(-\hat{a},0,0),\hat{M}_\sigma(0,0,\hat{b}),\hat{M}_\sigma(0,\hat{b},\hat{b})\}$ is a subset of $\hat{H}_{a,b}$. Explicitly we have \begin{align*}
        P = \{&(0,0,0),(-\hat{a}, 0, 0), (0, -\hat{a}, 0), (0, 0, -\hat{a}), (\hat{a}, \hat{a}, \hat{a}), (\hat{b}, 0, 0),(0, \hat{b}, 0), (0, 0, \hat{b}),\\ 
        &(0, \hat{b}, \hat{b}),(\hat{b}, 0, \hat{b}),(\hat{b}, \hat{b}, 0), (0, -\hat{b}, -\hat{b}),(-\hat{b}, 0, -\hat{b}), (-\hat{b}, -\hat{b}, 0), (-\hat{b}, -\hat{b}, -\hat{b})\}.
    \end{align*}
    We need to check that for $(x,y,z) \in P$ we have $(x,y,z) \in \hat{H}_{a,b}$, that is, we have $-\hat{b}\cdot x + \hat{a}\cdot z \leq \hat{a}\hat{b}$. This inequality holds simultaneously for the points $(-\hat{b},-\hat{b},0)$ and $(\hat{a},\hat{a},\hat{a})$ if and only $\hat{b} \leq \hat{a} \leq 2\hat{b}$. It can be seen that the inequalities obtained from the other points also hold for this regime of parameters. This concludes the proof that the left-hand side of (\ref{eq: intersection half-spaces}) is contained in the right-hand side for $q=4$, which is the final thing that we needed to show to prove the lemma.
\end{proof}

Lemma~\ref{lem: One region is enough} states that it is enough to understand $F(\mathcal{T}_{a,b} \cap \mathcal{R}_\sigma)$ for a specific $\sigma$ to understand the whole image $F(\mathcal{T}_{a,b})$. In the following two lemmas we calculate $\mathcal{T}_{a,b} \cap \mathcal{R}_\sigma$ more explicitly for $\sigma = (123)$ and $\sigma=(134)$ for $q=3$ and $q=4$ respectively. Because $F$ maps $\mathcal{R}_{(123)}$ into $\mathcal{R}_{(23)}$ for $q=3$ and $\mathcal{R}_{(134)}$ into $\mathcal{R}_{(243)}$ for $q=4$, we describe $\mathcal{T}_{a,b} \cap \mathcal{R}_\sigma$ for these instances of $\sigma$ too. The choice for these specific permutations $\sigma$ is arbitrary, but does seem to make the upcoming analysis more pleasant than for some other choices.

\begin{lemma}
    \label{lem: region equations}
    Let $a,b \in \mathbb{R}_{>1}$ and define 
    \[
        l_{a,b}(x) = b\cdot x^{\log(b)/\log(a)}
    \]
    For $q = 3$ we have 
    \[
        \mathcal{T}_{a,b} \cap \mathcal{R}_{(23)} = \{(x,y) \in \mathbb{R}_{>0}^2: x \leq 1\leq y \leq l_{a,b}(x)\}
    \]
    and
    \[
        \mathcal{T}_{a,b} \cap \mathcal{R}_{(123)} = \{(x,y) \in \mathbb{R}_{>0}^2: y \leq 1 \leq x \leq l_{a,b}(y)\}.
    \]
    For $q=4$ we have 
    \[
        \mathcal{T}_{a,b} \cap \mathcal{R}_{(243)} = \{(x,y,z) \in \mathbb{R}_{>0}^3: x \leq 1 \leq y \leq z \leq l_{a,b}(x)\}.
    \]
    and
    \[
        \mathcal{T}_{a,b} \cap \mathcal{R}_{(134)} = \{(x,y,z) \in \mathbb{R}_{>0}^3: z \leq y \leq 1 \leq x \leq y \cdot l_{a,b}(z/y)\}.
    \]
\end{lemma}
\begin{proof}
    We will first prove the statement for $q=3$. Recall that $\hat{\mathcal{T}}_{a,b} \cap \hat{\mathcal{R}}_{(23)} = \hat{\mathcal{H}}_{a,b}\cap \hat{\mathcal{R}}_{(23)}$, where $\hat{\mathcal{H}}_{a,b} = \{(x,y) \in \mathbb{R}^2: -\hat{b}\cdot x + \hat{a}\cdot y \leq \hat{a}\hat{b}\}$ and $\hat{\mathcal{R}}_{(23)} = \{(x,y) \in \mathbb{R}^2: x \leq 0 \leq y\}$. Therefore we can write
    \[
        \hat{\mathcal{H}}_{a,b}\cap \hat{\mathcal{R}}_{(23)}
        =
        \{(\hat{x},\hat{y})\in\mathbb{R}^2: \hat{x} \leq 0\leq \hat{y} \leq \frac{\hat{b}}{\hat{a}} \cdot \hat{x} + \hat{b}\}.
    \]
    If we replace $\hat{a}$ by $\log(a)$ and $\hat{b}$ by $\log(b)$
    we find that 
    \[
        \mathcal{T}_{a,b} \cap \mathcal{R}_{(23)} = 
        \exp(\hat{\mathcal{H}}_{a,b}\cap \hat{\mathcal{R}}_{(23)})
        =
        \{(e^{\hat{x}},e^{\hat{y}})\in\mathbb{R}_{>0}^2: \hat{x} \leq 0\leq \hat{y} \leq \frac{\log(b)}{\log(a)} \cdot \hat{x} + \log(b)\}.
    \]
    By applying $\exp$ to the individual components of the inequalities and replacing $x = e^{\hat{x}}$ and $y = e^{\hat{y}}$, we obtain the equality stated in the lemma. To prove the other equality for $q=3$ we note that for $\sigma = (12)$ we have $\sigma (23) = (123)$ and thus $M_\sigma(\mathcal{T}_{a,b} \cap \mathcal{R}_{(23)}) = \mathcal{T}_{a,b} \cap \mathcal{R}_{(123)}$. For $(x,y) \in \mathbb{R}_{>0}^2$ we have $M_{(12)}(x,y) = (y,x)$ and thus
    \begin{align*}
        \mathcal{T}_{a,b} \cap \mathcal{R}_{(123)} = \{(y,x) \in \mathbb{R}_{>0}^2: x \leq 1\leq y \leq l_{a,b}(x)\}
        = \{(x,y) \in \mathbb{R}_{>0}^2: y \leq 1 \leq x \leq l_{a,b}(y)\}.
    \end{align*}
    
    To prove the statements given for $q=4$ we recall that in that case $    \hat{H}_{a,b} = \{(x,y,z) \in \mathbb{R}^3: -\hat{b}\cdot x + \hat{a}\cdot z \leq \hat{a}\hat{b}\}$ and 
    $\hat{\mathcal{R}}_{(243)} = \{(x,y,z) \in \mathcal{R}^3: x \leq 0\leq y \leq z\}$. Therefore 
    \[  \hat{\mathcal{T}}_{a,b} \cap \hat{\mathcal{R}}_{(243)} = 
        \hat{\mathcal{H}}_{a,b}\cap \hat{\mathcal{R}}_{(243)} = \{(\hat{x},\hat{y},\hat{z})\in\mathbb{R}^3: \hat{x} \leq 0 \leq \hat{y} \leq \hat{z} \leq \frac{\hat{b}}{\hat{a}} \cdot \hat{x} + \hat{b}\}.
    \]
    Similarly, as in the $q=3$ case, it follows that 
    \[
        \mathcal{T}_{a,b} \cap \mathcal{R}_{(243)} 
        =
        \exp(\hat{\mathcal{T}}_{a,b} \cap \hat{\mathcal{R}}_{(243)})
        =
        \{(x,y,z) \in \mathbb{R}_{>0}^3: x \leq 1 \leq y \leq z \leq l_{a,b}(x)\}.
    \]
    If we let $\sigma = (13)(24)$ we have $\sigma \cdot (243) = (134)$. For this $\sigma$ and $(x,y,z) \in \mathbb{R}_{>0}^{3}$ we have
    $M_\sigma(x,y,z) = (z/y,1/y,x/y)$. We find that
    \begin{align*}
        \mathcal{T}_{a,b} \cap \mathcal{R}_{(134)} &= 
        M_\sigma(\mathcal{T}_{a,b} \cap \mathcal{R}_{(243)} )\\
        &=
        \{(z/y,1/y,x/y) \in \mathbb{R}_{>0}^3: x \leq 1 \leq y \leq z \leq l_{a,b}(x)\}\\
        &=
        \{(x,y,z) \in \mathbb{R}_{>0}^3: z/y \leq 1 \leq 1/y \leq x/y \leq l_{a,b}(z/y)\}\\
        &=
        \{(x,y,z) \in \mathbb{R}_{>0}^3: z \leq y \leq 1 \leq x \leq y\cdot l_{a,b}(z/y)\}.
    \end{align*}
\end{proof}

The next lemma provides inner and outer approximations of the sets $\mathcal{T}_{a,b}$ with simple polytopes.

\begin{lemma}
    \label{lem: convex hulls}
    Let $a,b \in \mathbb{R}_{>1}$ with $a \geq  b$. Then for q=3 we have
    \[
        \Conv(\{(1,1),(1/a,1),(1,b)\}) \subseteq \mathcal{T}_{a,b} \cap \mathcal{R}_{(23)}
    \]
    and
    \[
        \mathcal{T}_{a,b} \cap \mathcal{R}_{(123)} \subseteq 
        \Conv(\{(1,1),(b,1),(1,1-\frac{(b-1)\log(a)}{b\log(b)})\}).
    \]
    For q=4 we have
    \[
        \Conv(\{(1,1,1),(1/a,1,1),(1,b,b),(1,1,b)\}) \subseteq \mathcal{T}_{a,b} \cap \mathcal{R}_{(243)}
    \]
    and
    \[
        \mathcal{T}_{a,b} \cap \mathcal{R}_{(134)} \subseteq 
        \Conv(\{(1,1,1),(b,1,1),(1,1/b,1/b),(1,1,1-\frac{(b-1)\log(a)}{b\log(b)})\}).
    \]
\end{lemma}
\begin{proof}
    Let $l_{a,b}$ be as in Lemma~\ref{lem: region equations}. We define $c = \log(b)/\log(a)$ so that we can write $l_{a,b}(x) = b \cdot x^c$. By assumption $c\leq 1$, therefore the function $l_{a,b}$ is concave and thus the sets 
    \[
        \{(x,y) \in \mathbb{R}_{>0}^2: y \leq l_{a,b}(x)\} 
        \quad\text{ and }
        \quad
        \{(x,y,z) \in \mathbb{R}_{>0}^3: z \leq l_{a,b}(x)\} 
    \]
    are convex. It follows now from Lemma~\ref{lem: region equations} that the sets $\mathcal{T}_{a,b} \cap \mathcal{R}_{(23)}$ for $q=3$ and 
    $\mathcal{T}_{a,b} \cap \mathcal{R}_{(243)}$ for $q=4$ are convex. It is easy to see that the former set contains the points $(1,1),(1/a,1)$ and $(1,b)$ and that the latter set contains the points $(1,1,1),(1/a,1,1),(1,b,b)$ and $(1,1,b)$. This is enough to conclude that the first stated inclusions for $q=3$ and $q=4$ hold.
    
    Because $l_{a,b}$ is concave we find that for all $x>0$
    \begin{equation}
    \label{eq: concave inequality}
        l_{a,b}(x) \leq l_{a,b}'(1)(x-1) + l_{a,b}(1) = bc(x-1) + b. 
    \end{equation}
    Therefore, using Lemma~\ref{lem: region equations}, we have the following inclusion for $q=3$
    \[
         \mathcal{T}_{a,b} \cap \mathcal{R}_{(123)} = \{(x,y) \in \mathbb{R}_{>0}^2: y \leq 1 \leq x \leq l_{a,b}(y)\}
        \subseteq 
        \{(x,y) \in \mathbb{R}^2: y \leq 1 \leq x \leq bc(y-1)+b\}.
    \]
    Note that in the latter set we do not require $x$ and $y$ to be positive. This set can also be written as the intersection of the following three half-spaces
    \[
        H_1 = \{(x,y) \in \mathbb{R}^2: y \leq 1\},\
        H_2 = \{(x,y) \in \mathbb{R}^2: 1 \leq x\}
        \text{ and }
        H_3 = \{(x,y) \in \mathbb{R}^2: x \leq bc(y-1)+b\}.
    \]
    Note that $(1,1) \in \partial H_1 \cap \partial H_2 \cap \Int(H_3)$, 
    $(b,1) \in \partial H_1 \cap \Int(H_2) \cap \partial H_3$ and
    $(1,1-\frac{b-1}{bc}) \in \Int(H_1) \cap \partial H_2 \cap \partial H_3$. The second inclusion for $q=3$ stated in the lemma follows from Lemma~\ref{lem: half-spaces}.
    
    
    From Lemma~\ref{lem: region equations} and equation~{(\ref{eq: concave inequality})} we deduce that for q=4
    \[
        \mathcal{T}_{a,b} \cap \mathcal{R}_{(134)} \subseteq \{(x,y,z) \in \mathbb{R}^3: z \leq y \leq 1 \leq x \leq y \cdot \left(bc(z/y-1) + b\right)\}.
    \]
    So $\mathcal{T}_{a,b} \cap \mathcal{R}_{(134)}$ is contained in the intersection of the following half-spaces
    \begin{alignat*}{2}
    H_1 &= \{(x,y,z) \in \mathbb{R}^3: z \leq y\},
    \quad
    H_2 &&=\{(x,y,z) \in \mathbb{R}^3: y \leq 1\},\\
    H_3 &= \{(x,y,z) \in \mathbb{R}^3: 1 \leq x\},
    \quad
    H_4 &&=\{(x,y,z) \in \mathbb{R}^3: x \leq bc(z-y)+by\}.
    \end{alignat*}
    We see that 
    \begin{alignat*}{2}
    (1,1,1)&\in \partial H_1\cap \partial H_2 \cap \partial H_3 \cap \Int(H_4),
    \quad
    &&(b,1,1)\in \partial H_1\cap \partial H_2  \cap \Int(H_3) \cap \partial H_4\\
    (1,1/b,1/b)&\in \partial H_1  \cap \Int(H_2) \cap \partial H_3 \cap \partial H_4,
    \quad
    &&(1,1,1-\frac{b-1}{bc})\in \Int(H_1)\cap \partial H_2 \cap \partial H_3 \cap \partial H_4.
    \end{alignat*}
    Using Lemma~\ref{lem: half-spaces} we can conclude that the last stated inclusion in the lemma indeed holds.
\end{proof}

\end{subsection}

\begin{subsection}{Proof of the \hyperref[thm:main]{Main Theorem}}
\label{sec: main thm proof}

In this section we prove the \hyperref[thm:main]{Main Theorem}. We utilize a number of inequalities for which the proofs can be found in the next section.
\begin{lemma}\label{lem:onecondition}
    Let $q \in \{3,4\}$, $d \in \mathbb{Z}_{\geq 2}$ for $q=3$ and $d \in \mathbb{Z}_{\geq 4}$ for $q =4$ and let $1- q/(d+1) \leq w < 1$. 
    Let $a,b \in \mathbb{R}_{>1}$ such that 
    \begin{equation}\label{eq:inequality}
             \max\left\{ b, \frac{b^d + w + k-2}{w b^d + (q-1)}\right\}
        < 
        a
        < 
        b^{\frac{b^d(q-1+w)(b-1)}{(b^d-1)(b-w)}}.
     \end{equation}
    Then $F(\mathcal{T}_{a,b}) \subseteq \Int(\mathcal{T}_{a,b})$.
\end{lemma}
\begin{proof}
    Recall from Section~\ref{sec: Intro functions} that we can write $F$ as the 
    composition $G \circ P$, where $P(x_1,\dots, x_{q-1}) = (x_1^d, \dots, x_{q-1}^d)$. In logarithmic coordinates the map $\hat{P} = \log\circ P\circ\exp$
    acts as multiplication by $d$. In the proof of Lemma~\ref{lem:convexity} we showed that $\hat{\mathcal{T}}_{a,b}$
    is a polytope whose vertices have entries $0$, $\pm \hat{a}$ or $\pm \hat{b}$. It follows that $\hat{P}(\hat{\mathcal{T}}_{a,b})$ is the same polytope where $\hat{a}$ and $\hat{b}$ are replaced by $d\cdot\hat{a}$ and $d\cdot\hat{b}$ respectively. Because $\hat{a} = \log(a)$ and $\hat{b} = \log(b)$, we can conclude that $P(\mathcal{T}_{a,b}) = \mathcal{T}_{a^d,b^d}$. It follows from Lemma~\ref{lem: One region is enough} that it is enough to show that $G (\mathcal{T}_{a^d,b^d} \cap \mathcal{R}_{(123)}) = F(\mathcal{T}_{a,b} \cap \mathcal{R}_{(123)}) \subseteq \Int(\mathcal{T}_{a,b})$ for $q=3$ and
    $G (\mathcal{T}_{a^d,b^d} \cap \mathcal{R}_{(134)}) = F(\mathcal{T}_{a,b} \cap \mathcal{R}_{(134)}) \subseteq \Int(\mathcal{T}_{a,b})$ for $q=4$. 
    
    We use Lemma~\ref{lem: convex hulls} to conclude that it is enough to show that 
    \begin{equation}
    \label{eq: conv3}
        G\big(\Conv(\{(1,1),(b^d,1),(1,1-\frac{(b^d-1)\log(a)}{b^d\log(b)})\})\big) \subseteq \Int(\mathcal{T}_{a,b})
    \end{equation}
    for $q=3$ and
    \begin{equation}
    \label{eq: conv4}
        G\big(\Conv(\{(1,1,1),(b^d,1,1),(1,1/b^d,1/b^d),(1,1,1-\frac{(b^d-1)\log(a)}{b^d\log(b)})\})\big) \subseteq \Int(\mathcal{T}_{a,b})
    \end{equation}
    for $q=4$. We have to be careful here because initially we defined
    $G$ as a map on $\mathbb{R}_{>0}^{q-1}$. We can extend $G$ to the half-space $H = \{(x_1, \dots, x_{q-1}): x_1 + \cdots + x_{q-1} + w > 0\}$. To show that the sets in equations (\ref{eq: conv3}) and (\ref{eq: conv4}) are contained in $H$ it is enough to show that the vertices of these convex hulls are contained in $H$. This is clear for all but the last written vertex in either case. We will show that the equation $x_1 + \cdots + x_{q-1} + w > 0$ does indeed hold for these two points. Namely, by \eqref{eq:inequality} we have
    \begin{align*}
        x_1 + \cdots + x_{q-1} + w &= q-1 - \frac{(b^d-1)\log(a)}{b^d\log(b)} + w \\
        & > q-1 - \frac{(b^d-1)}{b^d}\cdot \frac{b^d(q-1+w)(b-1)}{(b^d-1)(b-w)} + w\\
        &= \frac{(1-w)(q-1+w)}{b-w} \geq 0,
    \end{align*}
    as desired.
    
    The map $G$ is a linear-fractional function, which means that $G$ sends line segments to line segments (see e.g. Section 2.3.3 of \cite{boyd2004convex}). Thus, for any set of points $p_1, \dots, p_n$ we have $G(\Conv(\{p_1,\dots, p_n\})) = \Conv{(\{G(p_1),\dots,G(p_n)\})}$. Let
    \[
        f_{q}(x) = \frac{w x + q-1}{x + k-2 + w}
        \quad\text{ and }\quad
        g(x) = \frac{(1+w)x+2}{2x+1+w}.
    \]
    The left-hand side of (\ref{eq: conv3}) is equal to
    \[
        \Conv\left(\{(1,1),(f_3(b^d),1),(1,f_{3}\left(1-\frac{(b^d-1)\log(a)}{b^d\log(b)}\right))\}\right)
    \]
    and the left-hand side of (\ref{eq: conv4}) is equal to
    \[
        \Conv(\{(1,1,1),(f_4(b^d),1,1),(1,g(1/b^d),g(1/b^d)),(1,1,f_4\left(1-\frac{(b^d-1)\log(a)}{b^d\log(b)}\right)).
    \]
    We can use Lemma~\ref{lem: convex hulls} to see that it is enough to show that
    \begin{equation}
        \label{eq: inequalities}
        f_q(b^d) > 1/a, 
        \quad
        g(1/b^d) < b
        \quad
        \text{ and }
        \quad
        f_q\left(1-\frac{(b^d-1)\log(a)}{b^d\log(b)}\right) < b
    \end{equation}
    to conclude that these sets are contained in $\mathcal{T}_{a,b} \cap \mathcal{R}_{(23)}$ and $\mathcal{T}_{a,b} \cap \mathcal{R}_{(243)}$ respectively. The first inequality follows directly from the assumptions. The second inequality follows from item~(\hyperref[it: inequality 4]{4}) of Theorem~\ref{thm:inequalities} below.
    For the last inequality we note that $f_q$ is strictly decreasing and $1-\frac{(b^d-1)\log(a)}{b^d\log(b)}$ is also strictly decreasing in $a$. Therefore, it is enough to show the inequality for $a= b^{\frac{b^d(q-1+w)(b-1)}{(b^d-1)(b-w)}}$. We obtain
    \[
        f_q\left(1-\frac{(b^d-1)\log(a)}{b^d\log(b)}\right) 
        <
        f_q\left(1-\frac{(q-1+w)(b-1)}{b-w}\right)
        =
        b.
    \]
    Because the inequalities in (\ref{eq: inequalities}) are strict we can even conclude that $\mathcal{T}_{a,b}$ gets mapped strictly inside itself by $F$, i.e. $F(\mathcal{T}_{a,b}) \subseteq \Int(\mathcal{T}_{a,b})$.
\end{proof}

\begin{theorem}\label{thm:precisemaintheorem}
Let $q \in \{3,4\}$, $d \in \mathbb{Z}_{\geq 2}$ for $q=3$ and $d \in \mathbb{Z}_{\geq 4}$ for $q =4$ and let $1- q/(d+1) \leq w < 1$ with $w>0$. Then the $q$-state Potts model with weight $w$ on the infinite $(d+1)$-regular tree, $\mathbb{T}_d$,  has a unique Gibbs measure.
\end{theorem}

\begin{proof}
    We will construct a sequence of subsets $\{\mathcal{T}_n\}_{n\geq 0}$ as is described in Lemma~\ref{lem: Tab sequence}.
    Define the functions 
    \[
        L(b) = \max\left\{b, \frac{b^d + w + k-2}{w b^d + (q-1)}\right\}
        \quad \text{ and }\quad
        U(b) = \min \left\{b^2,b^{\frac{b^d(q-1+w)(b-1)}{(b^d-1)(b-w)}}\right\}.
    \]
    It follows from items (\ref{it: inequality 1}), (\ref{it: inequality 2}) and (\ref{it: inequality 3}) of Theorem~\ref{thm:inequalities} that $L(b) < U(b)$ for $b > 1$. We define $M(b) = (L(b)+U(b))/2$ and note that $L(b) < M(b) < U(b)$ for $b>1$. Any element of $\mathbb{R}_{>0}^{q-1}$ is contained in $\mathcal{T}_{b,b}$ for a large enough value of $b$. It follows that we can choose $b_0>1$ such that $\mathcal{T}_{b_0,b_0}$ contains both the vector with every entry equal to $1/w$ and the vectors obtained from the all-ones vector with a single entry changed to $w$. Because $M(b_0)>b_0$ we have $\mathcal{T}_{b_0,b_0} \subset \mathcal{T}_{M(b_0),b_0}$ and thus $\mathcal{T}_{M(b_0),b_0}$ contains these vectors too. Inductively we now define $b_n$ for $n\geq 1$ by
    \[
        b_n = \inf\left\{b: F(\mathcal{T}_{M(b_{n-1}),b_{n-1}}) \subseteq \mathcal{T}_{M(b),b}\right\}.
    \]
    Because $\mathcal{T}_{M(b),b}$ moves continuously with $b$ it follows from Lemma~\ref{lem:onecondition} that $\{b_n\}_{n \geq 1}$ is a strictly decreasing sequence. The sequence is clearly bounded below by $1$ and thus it must have a limit. We claim that this limit is $1$. For the sake of contradiction assume that it has a limit $b_\infty > 1$. The set $\mathcal{T}_{M(b_\infty),b_\infty}$ gets mapped strictly inside itself by $F$ and thus there is a $b'<b_\infty$ such that $\mathcal{T}_{M(b_\infty),b_\infty}$ also gets mapped strictly inside $\mathcal{T}_{M(b'),b'}$. This is an open condition, so there is an $\epsilon>0$ such that $\mathcal{T}_{M(b),b}$ gets mapped strictly inside $\mathcal{T}_{M(b'),b'}$ for all $b \in [b_\infty,b_\infty + \epsilon)$. There must be an integer $N$ such that $b_N \in [b_\infty,b_\infty + \epsilon)$, but then $b_{N+1} < b' < b_\infty$, so $b_\infty$ cannot be the limit of the decreasing sequence $\{b_n\}_{n \geq 0}$.
    
    We define $\mathcal{T}_n = \mathcal{T}_{M(b_n),b_n}$. We have $b < M(b) < b^2$, so it follows from Lemma~\ref{lem:convexity} that every $\mathcal{T}_n$ is log-convex. We have chosen $\mathcal{T}_0$ such that condition (\ref{it:basecase}) of Lemma~\ref{lem: Tab sequence} is satisfied. By construction $F(\mathcal{T}_{m}) \subseteq \mathcal{T}_{m+1}$ for all $m$ and thus condition (\ref{it:inductionstep}) of Lemma~\ref{lem: Tab sequence} is satisfied. Finally, because both $b_n$ and $M(b_n)$ converge to $1$, it follows that the sequence of sets $\mathcal{T}_n$ converges to the set consisting of just the all-ones vector. This means that condition (\ref{it:laststep}) of Lemma~\ref{lem: Tab sequence} is satisfied. We can conclude that $\mathbb{T}_d$ has a unique Gibbs measure.
\end{proof}

\begin{remark}
The assumption $w>0$ is critical in the case $q=3$ and $d=2$, as it is well known there are multiple Gibbs measures at $w_c=0$ when $q=d+1$. 
One sees this in our argument as well. For the base case of the induction, condition (\ref{it:basecase}) of Lemma~\ref{lem: Tab sequence}, we need $\mathcal{T}_0$ to contain the vectors $(1,w) =(1,0)$, $(w,1)=(0,1)$ and $(1/w,1/w) = (\infty,\infty)$. If we take the log convex hull of these vectors and apply $F$, we obtain a region that again contains the vectors $(1,w) =(1,0)$, $(w,1)=(0,1)$ and $(1/w,1/w) = (\infty,\infty)$. It is thus possible to choose boundary conditions that yield unbounded ratios at an arbitrary distance from the leaves.
This observation is closely related to the existence of so-called frozen colorings \cite{BRIGHTWELL2000141}. These give distinct \emph{trivial} Gibbs measures, each supported on a single coloring of $\mathbb{T}_2$.



\end{remark}

\end{subsection}
\end{section}

\begin{section}{Proof of the inequalities}
\label{sec: Inequality section}
This section is dedicated to showing all the inequalities from the previous section are satisfied. We define the following functions
\begin{alignat*}{3}
l(q,d,w,b) &= \frac{ b^d+q-2+w}{w b^d + q-1}, \quad
   &&g(d,w,b) = \frac{2b^d+1+w}{(1+w)b^d+2},\\
    h(q,d,w,b) &=\frac{ b^d (b-1) (q-1+w)}{\left(b^d-1\right) (b - w)},  \quad \quad
    &&u(q,d,w,b) = b^{h(q,d,w,b)}.
\end{alignat*}
We mostly consider these as functions in $b$ and consider only $b \geq 1$.
Note that $h(q,d,w,b)$ has a removable singularity in $b=1$ with $h(q,d,w,1) = \frac{q-1+w}{d(1-w)}$. The main theorem we prove in this section is the following.
\begin{theorem}\label{thm:inequalities}
For $q=3, d \geq 2$ and $w \in [1 - \frac{3}{d+1}, 1)$ or for $q=4, d \geq 4$ and $w \in [1 - \frac{4}{d+1}, 1)$ we have for each $b >1$
\begin{enumerate}
    \item
    \label{it: inequality 1}
    $u(q,d,w,b) > l(q,d,w,b)$,
    \item
    \label{it: inequality 2}
    $u(q,d,w,b) > b$,
    \item
    \label{it: inequality 3}
    $b^2>l(q,d,w,b)$.
\end{enumerate}
And for all $b >1$ and $d \geq 3$ and $w \in [1 - \frac{4}{d+1}, 1)$ we have
\begin{enumerate}
    \item[(4)]\label{it: inequality 4}
    $g(d,w,b) < b$.
\end{enumerate}
\end{theorem}

In the next section we show it is enough to prove Theorem \ref{thm:inequalities} holds for $w = w_c = 1 - \frac{q}{d+1}$ where we take $q=4$ in inequality (\hyperref[it: inequality 4]{4}).
Subsequently, inequality (\ref{it: inequality 2}) is proved in Corollary \ref{cor:ineq2}, inequality (\ref{it: inequality 3}) is proved in Lemma \ref{lem:ineq3} and inequality (\hyperref[it: inequality 4]{4}) is proved in Lemma \ref{lem:ineq4}.
The proof of inequality (\ref{it: inequality 1}) is the most involved and is the result of Lemma \ref{lem: inequality lemma} and Lemma \ref{lem:ineq1}.

\begin{subsection}{Reduction to $w = w_c$}

\begin{lemma}\label{lem:reductiontowc}
    Let $q \geq 2, d \geq 1$ and $w \in [0,1)$. For $b>1$ we have $l(q,d,w,b)$ and  $g(d,w,b)$ are decreasing in $w$, while $u(q,d,w,b)$ is increasing in $w$. 
\end{lemma}
\begin{proof}
We compute 
\begin{align*}
   \frac{\partial}{\partial w} l(q,d,w,b)&= 
   - \frac{(b^d-1)(b^d+q-1)}{(wb^d+q-1)^2}
   ,\\
    \frac{\partial}{\partial w} g(d,w,b) &= -\frac{2(b^{2d}-1)}{((1+w)b^d+2)^2},\\
    \frac{\partial}{\partial w} u(q,d,w,b) &=u(q,d,w,b) \cdot \frac{b^d(b-1)(b+q-1)\log{b}}{(b^d-1)(b-w)^2}.
\end{align*}
We see that for $b>1$ we have $\frac{\partial}{\partial w} l(q,d,w,b)<0$, $\frac{\partial}{\partial w} g(d,w,b) < 0$ and $\frac{\partial}{\partial w} u(q,d,w,b)>0$, so the lemma follows.
\end{proof}

From Lemma \ref{lem:reductiontowc} it follows that if we can show Theorem \ref{thm:inequalities} holds for $w = w_c$, then it also holds for all $ w\in [w_c,1)$. So from now on we will work with $l(b,w_c,d,q)$, $u(b,w_c,d,q)$ and $h(b,w_c,d)$. To shorten notation we write
\begin{alignat*}{2}
    l(b) &= \frac{(d+1) b^d+d(q-1)-1}{(d-q+1) b^d+(d+1) (q-1)}, \quad &&g(b) = \frac{(d+1)b^d + d-1}{(d-1)b^d+d+1},\\
    h(b) &=\frac{d q b^d (b-1)}{\left(b^d-1\right) ((d+1)(b-1)+q)},\quad
   &&u(b) = b^{h(b)}.
\end{alignat*}
We note that the function $h$ has a removable singularity in $1$ with $h(1)=1$.
\end{subsection}

\begin{subsection}{Inequalities $g(b) < b$, $u(b) > b$ and $b^2 > l(b)$}

We will start by showing $g(b) <b$ holds for $b > 1$ and $d \geq 2$.
\begin{lemma}\label{lem:ineq4}
Let $d \geq 2$ and $b>1$. Then we have
$
g(b) < b.
$
\end{lemma}
\begin{proof}
We have $g(1) = 1$ and $g'(1) = 1$. Furthermore, one can see
\[
g''(b) = -\frac{4 d^2 \left(d^2-1\right) \left(b^d-1\right) b^{d-2}}{\left((d-1) b^d+d+1\right)^3} < 0
\]
for $d \geq 2$ and $b > 1$. This implies $g(b) < b$ for $d \geq 2$ and $b > 1$.
\end{proof}

Next we show that $h$ is increasing in $b$. This fact will immediately give us inequality (\ref{it: inequality 2}).  Furthermore, it is also helpful in proving a sufficient condition for inequality (\ref{it: inequality 1}) to hold, see Lemma \ref{lem: inequality lemma} below.
\begin{lemma}
\label{lem: h(b) strictly increasing}
For all  $b > 1,d\geq 2$ and $q \geq 2$ we have $h'(b)> 0$.
\end{lemma}
\begin{proof}
We compute 
\[
h'(b)= \frac{k d b^{d-1} \left(q b^{d+1}- d\left(d+1\right)b^2+ \left(2 d^2-d q+2 d-q\right)b-d^2+d q-d\right)}{\left(b^d-1\right)^2 ((d+1)(b-1)+q)^2}.
\]
It suffices to show that
\[
m(b) := q b^{d+1}- d\left(d+1\right)b^2+ \left(2 d^2-d q+2 d-q\right)b-d^2+d q-d
\]
is positive for $b>1$.
We compute
\begin{align*}
    m'(b) &= (d+1)(k(b^d-1)-2d(b-1)),\\
    m''(b) &= d(d+1)(kb^{d-1}-2).
\end{align*}
We see $m''(b) > 0$ for $b > 1,d\geq 2$ and $q \geq 2$. Noting that $m'(1) = 0$ and $m(1) = 0$, it follows that $m'(b)$ and $m(b)$ are strictly positive for $b>1$.
\end{proof}

This immediately implies inequality (\ref{it: inequality 2}). 

\begin{corollary}\label{cor:ineq2}
For $b>1$ we have $u(b) > b$.
\end{corollary}
\begin{proof}
Recall $u(b) = b^{h(b)}$. As $h(1) = 1$ and $h'(b)> 0$ for $b >1$ by Lemma \ref{lem: h(b) strictly increasing}, we see $u(b) > b$ for $b >1$ follows.
\end{proof}

Until this point, we did not need to assume $q=3$ or $q=4$ for the computations to work, but for inequality (\ref{it: inequality 3}) to hold we do need some restrictions on $q$ and $d$.

\begin{lemma}\label{lem:ineq3}
    For $q=3$ and $d \geq 2$ and for $q = 4$ and $d \geq 4$ we have
  $
         l(b) < b^2,
    $
    for all $b>1$.
\end{lemma}

\begin{proof}
Multiplying both sides of the inequality with the positive factor $(d-q+1)b^d+(d+1) (q-1)$ we obtain the equivalent inequality
\[
(d+1) b^d+d(q-1)-1 < (d-q+1)b^{d+2} +(d+1) (q-1)b^2.
\]
To show that this inequality holds we show that the polynomial
\[
Q(b) = (d-q+1)b^{d+2} - (d+1)b^d +(d+1) (q-1)b^2- d(q-1) + 1
\]
is strictly positive for $b>1$. For $d\geq 2$ we compute
\begin{align*}
    Q'(b) &= (d+2)(d-q+1)b^{d+1} - (d+1)db^{d-1} +2(d+1)(q-1)b,\\
    Q''(b) &= (d+2)(d+1)(d-q+1)b^{d} - (d+1)d(d-1)b^{d-2} +2(d+1)(q-1),\\
    Q'''(b) &=d(d+1)b^{d-3} \left((d+2)(d-q+1)b^2 - (d-1)(d-2) \right),
\end{align*}
Because $d+1 \geq k$ we find that for all $b \geq 1$
\[
    (d+2)(d-q+1)b^2 - (d-1)(d-2) \geq (d+2)(d-q+1) - (d-1)(d-2) = (6-q)d - 2q.
\]
For $q=3$ this quantity is nonnegative for $d \geq 2$ and for $q=4$ this quantity is nonnegative for $d \geq 4$. So in our case we can conclude that $Q'''(b) \geq 0$ for all $b \geq 1$. As we have $Q''(1) =  d(d+1) > 0$, $Q'(1) = 3d>0$ and $Q(1)  = 0$, it follows that $Q''(b),Q'(b)$ and $Q(b)$ are strictly positive for $b>1$.
\end{proof}

\end{subsection}

\begin{subsection}{The inequality $u(b) > l(b)$}

The following lemma contains a sufficient condition to prove this inequality. In the remainder of the section we prove that this condition is satisfied.

\begin{lemma}
\label{lem: inequality lemma}
Suppose for all $b>1$ we have
\begin{equation}
    \label{it: condition 3}
    \frac{l'(b)}{l(b)} < \frac{h(b)}{b} + 2\frac{b-1}{b+1}g'(b).
\end{equation}
Then $u(b) > l(b)$ for all $b > 1$.
\end{lemma}

\begin{proof}

As $l(b)$ and $u(b)$ are strictly positive for $b \geq 1$, we can define 
    \[
        F(b) = \log(u(b))-\log(l(b))
    \]
    for $b \geq 1$.
   Then we have
    \[
    F'(b) =  \frac{u'(b)}{u(b)} - \frac{l'(b)}{l(b)} = \frac{h(b)}{b} + \log(b) h'(b)
     - \frac{l'(b)}{l(b)}.
    \]
    For $b>1$ we have that $h'(b)> 0$ by Lemma \ref{lem: h(b) strictly increasing} and $\log(b) > 2(b-1)/(b+1)$, therefore 
    \[
    F'(b) >  \frac{h(b)}{b} + 2\frac{b-1}{b+1} h'(b)
     - \frac{l'(b)}{l(b)},
    \]
    which is positive by (\ref{it: condition 3}). It is easy to see $F'(1) = 0$. 
    Hence $F$ has a global minimum in $b=1$. 
    As $F(1) = 0$, it follows that $u(b) > l(b)$ for all $b > 1$, which is what we wanted to show.
\end{proof}

This lemma is useful because proving the inequality $u(b) > l(b)$ for all $b > 1$ can now be reduced to proving inequalities involving rational functions and with some work to inequalities involving only polynomials. The next lemma shows that (\ref{it: condition 3}) holds. For this to work we do need to restrict to $q=3$ and $d \geq 2$ or $q=4$ and $d \geq 4$.

\begin{lemma}\label{lem:ineq1}
For $q=3$ and  $d\geq 2$ and for $q=4$ and $d\geq 4$ and any $b>1$ we have
    \[
    \frac{l'(b)}{l(b)} < \frac{h(b)}{b} + 2\frac{b-1}{b+1}h'(b).
    \]
\end{lemma}

\begin{proof}

We introduce the following polynomials
\begin{alignat*}{2}
    p(b) &= (d+1) b^d+d(q-1)-1, \quad &&q(b) = (d-q+1) b^d+(d+1) (q-1),\\
    s(b) &= d q b^d (b-1), \quad &&t(b) = \left(b^d-1\right) ((d+1)(b-1)+q).
\end{alignat*}
Thus $l(b) = p(b)/q(b)$ and $h(b)=s(b)/t(b)$. Furthermore, we define
$r(b) = q(b)p'(b) - p(b)q'(b)$ and $v(b) = t(b)s'(b)-s(b)t'(b)$.  It is worth noting that $r(b)$ simplifies to $q^2 d^2 b^{d-1}$. The inequality we want to prove can now be written as 
\[
    \frac{r(b)}{p(b)q(b)} < \frac{s(b)}{b \cdot t(b)} + 2 \frac{(b-1)v(b)}{(b+1)t(b)^2}.
\]
For $b > 1$ the quantity $b(b+1)p(b)q(b)t(b)^2$ is strictly positive and thus it is equivalent to prove the inequality, where we have multiplied both sides by this term. We see that it is enough to prove that the following polynomial is strictly positive for all $b>1$
\begin{equation}
    \label{eq: definition P(b)}
    P(b) = (b+1)s(b)p(b)q(b)t(b) + 2b(b-1)v(b)p(b)q(b) - b(b+1)r(b)t(b)^2.
\end{equation}
It can be checked that the terms $s(b)$, $b\cdot v(b)$ and $b\cdot r(b)$ all contain a factor $qdb^d$ 
and thus $P_0(b) = P(b)/(kdb^d)$ is a polynomial in $b$ whose coefficients are polynomials in $d$. 
The remainder of the proof will be dedicated to showing that $P_0(b)$ is strictly positive for $b>1$. 

To avoid ambiguity later, we prove this for $q=3$ in the two cases $d=2$ and $d=3$ separately. For $d=2$ we have 
\[
    P_0(b) = 54 (b-1)^6+54 (b-1)^5
\]
and for $d=3$ we have
\begin{align*}
    P_0(b) =& 16 (b-1)^{12}+212 (b-1)^{11}+1236 (b-1)^{10}+4116 (b-1)^9+8793 (b-1)^8+12789 (b-1)^7+\\&12123 (b-1)^6+6318 (b-1)^5+1458 (b-1)^4.
\end{align*}
In both cases all the coefficients of $P_0(b)$ are strictly positive when written as a polynomial in $b-1$ and thus the polynomials are strictly positive for $b>1$.

We will now assume that $d \geq 4$. It can be seen by cross-multiplying the terms in the individual polynomials in (\ref{eq: definition P(b)}) that the only coefficients of $P_0(b)$ that can be non-zero appear in the $b^{i\cdot d+j}$ terms where $i,j \in \{0,1,2,3\}$. The exact coefficients are recorded in Table~\ref{tab: Coefficients of P}. For $n \in \{1,2,3\}$ we inductively define the polynomials $P_n(b) = P_{n-1}^{(4)}(b)/b^{d-4}$. Note that in this way $P_n(b)$ is a polynomial whose only non-zero coefficients appear in the $b^{i\cdot d + j}$ term, where $0\leq i \leq 3-n$ and $j \in \{0,1,2,3\}$.

\begin{table}[h!]
  \centering
  \caption{The coefficients of $P_0(b)$ for $d\geq 4$.}
  \label{tab: Coefficients of P}
    \begin{tabular}{l c r}
    \hline
  Term of $P_0(b)$ & Coefficient $q=3$ & Coefficient $q=4$ \\
    \hline
 $b^0$ & $(d-2) \left(8 d^3-3 d^2+2\right)$ & $(d-3) \left(18 d^3-d^2+3\right)$  \\
$b^1$ & $-24 d^4+19 d^3+60 d^2-24 d-14$  & $-54 d^4+67 d^3+173 d^2-39 d-27$ \\
$b^2$ & $(d+1) (3 d-1) \left(8 d^2+d-16\right)$ &  $(d+1) \left(54 d^3-23 d^2-124 d+33\right)$ \\
$b^3$ & $-(d+1)^2 \left(8 d^2+3 d-2\right)
   $  &  $-(d+1)^2 \left(18 d^2+7 d-3\right)$ \\
   \hline
   $b^{d}$ & $(d-2) \left(8 d^3+4 d^2-d-6\right)$ &  $(d-3) \left(12 d^3+3 d^2-2 d-9\right)$ \\
   $b^{d+1}$ & $-3 \left(8 d^4-8 d^3+15 d^2-19 d-14\right)$ & $-36 d^4+73 d^3-129 d^2+99 d+81$ \\
   $b^{d+2}$ &$3 \left(8 d^4-4 d^3+15 d^2+20 d-16\right)$ & $36 d^4-47 d^3+115 d^2+163 d-99$ \\
   $b^{d+3}$ & $-(d+1) \left(8 d^3-8
   d^2-d+6\right) $ &  $-(d+1) \left(12 d^3-19 d^2-6 d+9\right)$ \\
   \hline
   $b^{2d}$ & $2 (d-2) (d+1) \left(d^2-2 d+3\right)$ &  $(d-3) (d+1) \left(2 d^2-5 d+9\right)$ \\
   $b^{2d+1}$ & $-3 \left(2 d^4-4 d^3+3 d^2+14 d+14\right)$ & $-6 d^4+21 d^3-37 d^2-81 d-81$ \\
   $b^{2d+2}$ & $3 \left(2 d^4-2 d^3-7 d+16\right)$ & $6 d^4-15 d^3+19 d^2-53 d+99$ \\
   $b^{2d+3}$ & $-(d-2) (d+1) \left(2 d^2+2
   d+3\right) $ & $-(d-3) (d+1) \left(2 d^2+d+3\right)$ \\
   \hline
   $b^{3d}$ & $(d-2)^2 (d+1)$ & $(d-3)^2 (d+1) $ \\
   $b^{3d+1}$ & $-(d-2) (d+1) (d+7)$ & $-(d-3) (d+1) (d+9)$ \\
   $b^{3d+2}$ &$-(d-8) (d-2) (d+1)$ & $-(d-11) (d-3) (d+1)$ \\
   $b^{3d+3}$ &$(d-2) (d+1)^2 $ & $(d-3) (d+1)^2$ \\
  \hline
  \end{tabular}
\end{table}

The values of $P_{j}^{(i)}(1)$ as a polynomial in $x = d-4$, up to a common positive multiplicative factor, for $q=3$ and $q=4$ are contained in tables~\ref{tab: Values of P_ij q=3}~and~\ref{tab: Values of P_ij q=4} respectively. These polynomials have only nonnegative coefficients, from which it follows that their values are nonnegative for all $d\geq 4$.

The polynomial $P_3(b)$ is a cubic polynomial and thus its third derivative $P_3^{(3)}(b)$ is constant. 
Its exact value, which is recorded in Table~\ref{tab: Values of P_ij q=3} for $q=3$ and in Table~\ref{tab: Values of P_ij q=4} for $q=4$, is strictly positive for all $x \geq 0$, i.e. for all $d \geq 4$. We claim that it now follows inductively that $P_j^{(i)}(b)$ is strictly positive for all $b > 1$. Namely, suppose that for $i\in \{0,1,2,3\}$ we have shown that $P_{j}^{(i+1)}(b)$ is strictly positive for $b>1$. Then it follows that $P_{j}^{(i)}(b)$ is strictly increasing. Because $P_{j}^{(i)}(1) \geq 0$ (cf. Table~\ref{tab: Values of P_ij q=3} and Table~\ref{tab: Values of P_ij q=4}), we can conclude from this that $P_{j}^{(i)}(b)$ is also strictly positive for $b>1$. Furthermore, if $P_{j+1}(b) > 0$ for $b>1$ then the same follows for $P_{j}^{(4)}$ because $b^{d-4}\cdot P_{j+1}(b) = P_{j}^{(4)}(b)$. In conclusion, it follows that $P_0(b)>0$ for $b>1$, which is what we set out to prove.

\end{proof}

\begin{table}[h!]
  \centering
  \caption{The values of $P_j^{(i)}(1)$ for $q=3$ in the variable $x=d-4$ divided by  $6 (x+4)^3 (x+5)$ for $i,j \in \{0,1,2,3\}$.}
  \label{tab: Values of P_ij q=3}
  \resizebox{\textwidth}{!} {%
  \begin{tabular}{l r}
    \hline
    $P_0 (1)$  & 0 \\ 
    $P_0^{(1)} (1)$  & 0 \\ 
    $P_0^{(2)} (1)$  & 0 \\ 
    $P_0^{(3)} (1)$  & 0 \\ 
    \hline
    $P_1 (1)$  & $54 (x+2)$ \\ 
    $P_1^{(1)} (1)$  & $3 \left(122 x^2+759 x+1045\right)$ \\
    $P_1^{(2)} (1)$  & $3 \left(478 x^3+5019 x^2+16831 x+17560\right)$ \\
    $P_1^{(3)} (1)$  & $4276 x^4+61731 x^3+328134 x^2+754415 x+623616$ \\
    \hline
    $P_2 (1)$  &  $4 \left(2864 x^5+51218 x^4+363231 x^3+1272211 x^2+2188942 x+1467858\right)$ \\ 
    $P_2^{(1)} (1)$  & $2 \left(8800 x^6+200624 x^5+1895748 x^4+9479789 x^3+26371144 x^2+38515725 x+22913226\right)$ \\
    $P_2^{(2)} (1)$  & $4 \left(6100 x^7+166078 x^6+1935943 x^5+12502085 x^4+48198140 x^3+110605547 x^2+139341417 x+73916010\right) $ \\
    $P_2^{(3)} (1)$  & $ 4 (x+2) \left(7948 x^7+229772 x^6+2871108 x^5+20093453 x^4+85033465 x^3+217534941 x^2+311415975 x+192411450\right) $ \\
    \hline
    $P_3 (1)$  & $12 (x+2) \left(3324 x^8+105498 x^7+1478477 x^6+11945536 x^5+60841362 x^4+199973638 x^3+414113609 x^2+493884000 x+259667100\right)$ \\
    $P_3^{(1)} (1)$  & $4 (x+2) (x+5) (2 x+9) (3 x+13) \left(372 x^6+10607 x^5+124569 x^4+775749 x^3+2712487 x^2+5063412 x+3950100\right) $ \\ 
    $P_3^{(2)} (1)$  &  $4 (x+2) (x+5)^2 (x+6) (2 x+9) (3 x+13) (3 x+14) \left(12 x^4+368 x^3+3431 x^2+13148 x+18249\right)$ \\
    $P_3^{(3)} (1)$  & $36 (x+2) (x+5)^4 (x+6) (x+7) (2 x+9) (2 x+11) (3 x+13) (3 x+14)$ \\
    \hline
  \end{tabular}}
\end{table}

\begin{table}[h!]
  \centering
  \caption{The values of $P_j^{(i)}(1)$ for $q=4$ in the variable $x=d-4$ divided by  $8 (x+4)^3 (x+5)$ for $i,j \in \{0,1,2,3\}$.}
  \label{tab: Values of P_ij q=4}
  \resizebox{\textwidth}{!} {%
  \begin{tabular}{l r}
    \hline
    $P_0 (1)$  & 0 \\ 
    $P_0^{(1)} (1)$  & 0 \\ 
    $P_0^{(2)} (1)$  & 0 \\ 
    $P_0^{(3)} (1)$  & 0 \\ 
    \hline
    $P_1 (1)$  & $48 x  $ \\ 
    $P_1^{(1)} (1)$  & $ 8 \left(44 x^2+215 x+135\right)$ \\
    $P_1^{(2)} (1)$  & $2 \left(709 x^3+6684 x^2+18585 x+12690\right)$ \\
    $P_1^{(3)} (1)$  & $ 4134 x^4+55427 x^3+265045 x^2+515121 x+308889$ \\
    \hline
    $P_2 (1)$  &  $4 \left(2699 x^5+45392 x^4+297314 x^3+933894 x^2+1366575 x+695142\right)$ \\ 
    $P_2^{(1)} (1)$  & $2 \left(8020 x^6+174193 x^5+1549849 x^4+7170108 x^3+17941968 x^2+22435155 x+10317699\right) $ \\
    $P_2^{(2)} (1)$  & $2 \left(10878 x^7+284024 x^6+3149973 x^5+19136364 x^4+68251666 x^3+141154110 x^2+153239211 x+64122030\right) $ \\
    $P_2^{(3)} (1)$  & $3 (x+1) \left(9316 x^7+267882 x^6+3329185 x^5+23169850 x^4+97489094 x^3+247912018 x^2+352706325 x+216527850\right) $ \\
    \hline
    $P_3 (1)$  & $12 (x+1) \left(2889 x^8+91143 x^7+1269517 x^6+10192836 x^5+51576597 x^4+168375593 x^3+346232169 x^2+409934700 x+213929100\right)$ \\
    $P_3^{(1)} (1)$  & $3 (x+1) (x+5) (2 x+9) (3 x+13) \left(408 x^6+11669 x^5+136865 x^4+848735 x^3+2949267 x^2+5463996 x+4227300\right) $ \\ 
    $P_3^{(2)} (1)$  &  $6 (x+1) (x+5)^2 (x+6) (2 x+9) (3 x+13) (3 x+14) \left(6 x^4+193 x^3+1807 x^2+6883 x+9471\right)$ \\
    $P_3^{(3)} (1)$  & $27 (x+1) (x+5)^4 (x+6) (x+7) (2 x+9) (2 x+11) (3 x+13) (3 x+14)$ \\
    \hline
  \end{tabular}}
\end{table}


\end{subsection}

We have shown the inequalities (\ref{it: inequality 1}), (\ref{it: inequality 2}), (\ref{it: inequality 3}) all hold, thus the conditions in Lemma \ref{lem:onecondition} and Lemma \ref{lem:convexity} are satisfied.
\end{section}

\begin{section}{Concluding remarks}\label{sec:conclude}
We conclude with some remarks concerning the possibility of expanding our approach and with some questions.
\\

\paragraph{\bf Generalisation}
The biggest challenge to generalizing our method to other values of $(q,d)$ comes from the fact that inequality \eqref{it: inequality 3} from Theorem~\ref{thm:inequalities} is not necessarily true for all $b>1$. This suggests that it might not be possible in all cases to find arbitrarily large log convex regions that get mapped into themselves. We suspect that in general this is indeed impossible when one requires the regions to have the symmetry that we use in this paper, that is regions $\mathcal{T} \subset \mathbb{R}^{q-1}_{>0}$ with $M_\sigma(\mathcal{T}) = \mathcal{T}$ for all $\sigma \in S_q$. A consequence is that in some cases we cannot make the region large enough to start the induction laid out in Lemma \ref{lem: Tab sequence}. Fortunately, inequality \eqref{it: inequality 3} does hold near $b>1$ for all $(q,d)$ with $d \geq q-1$ and $w \geq w_c$. This suggests that, at least when $w_c$ is close to $1$, i.e. when $d$ is large enough compared to $q$, our methods could still be applied. Moreover, it might be possible to find a separate argument to show that the ratios of $\hat{\mathbb{T}}_d^n$ get at least moderately close to $1$ for some $n$. This could then be used to bootstrap the induction in Lemma \ref{lem: Tab sequence}.

There are two more complications that prevent us from applying our method directly to other values of $(q,d)$. We suspect that these can be overcome with more thorough analysis. The first one comes from the fact that inequality \eqref{it: inequality 1} from Theorem~\ref{thm:inequalities} is no longer satisfied for most values of $(q,d)$ and $w_c$. The precise form of this inequality highly depends on our method of proof and specifically on our choice of upper bound for $l_{a,b}$ in equation \eqref{eq: concave inequality}. Computer analysis suggests that by taking different upper bounds for $l_{a,b}$, specifically taking tangent lines at different points, the proof that $\mathcal{T}_{a,b}$ gets mapped into itself for $a$ and $b$ near $1$ can be salvaged. The other complication appears when $q\geq5$. In this case one obtains more inequalities analogous to inequality~{(\hyperref[it: inequality 4]{4})} from Theorem~\ref{thm:inequalities}. These are not all satisfied when we take a naive generalization of the region $\mathcal{T}_{a,b}$. We suspect that this can be remedied by letting the regions depend on more than just two parameters. This leads to the analysis specifically of the log convexity of the regions becoming more involved.
\\





\paragraph{\bf The case  $(q,d)=(4,3)$} 
Unfortunately, our approach does not allows us to handle the case $(q,d)=(4,3)$. We briefly explain the complications.
In inequality \eqref{eq: concave inequality} we use the tangent line of $l_{a,b}(x)$ at $x=1$ to upper bound $l_{a,b}(x)$; this makes the calculus easier and this choice works for $q=4$ and $d\geq 4$.
We have evidence that by using the tangent line at a different point in inequality \eqref{eq: concave inequality} the calculations that follow from this upper bound also work for the case $q=4$ and $d=3$. However, we can show that inequality \eqref{it: inequality 3} in Theorem~\ref{thm:inequalities} fails when $(q,d)=(4,3)$ and $b >1$ is large enough, meaning that in that case the set $\mathcal{T}_{a,b}$ cannot both be log convex and satisfy $F(\mathcal{T}_{a,b})\subset\mathcal{T}_{a,b}$. For $b>1$ close enough to $1$ inequality \eqref{it: inequality 3} in Theorem~\ref{thm:inequalities} does hold. We suspect that our approach can be tweaked to show uniqueness for all $w \in (0,1)$ when $q=4$ and $d=3$, possibly by finding a separate argument to show that the ratios of $\hat{\mathbb{T}}_3^n$ get at least moderately close to $1$ for some $n$, bootstrapping the induction in Lemma \ref{lem: Tab sequence}.
\\

\paragraph{\bf Zero-free region}
Our final comment is related to the following question. Given $q\in \mathbb{N}$ and $d\geq q-1$. Does there exist a region $U$ in $\mathbb{C}$ containing the interval $(1-\frac{q}{d+1},1]$ such that for any $w\in U$ and any graph $G$ of maximum degree $d+1$ the partition function $Z(G,q,w)\neq 0$? (If so this would yield an efficient algorithm for approximately computing $Z(G,q,w)$ in this region by Barvinok's method~\cite{barvinok} combined with~\cite{PR}.)

Following~\cite{bencs}, to prove this, for say $q=3$, we would essentially need to find a log convex set $S\subsetneq\mathbb {C}^2$ such that the map $F$ maps $S$ into $S$ and such that $S$ satisfies some additional properties that we will not discuss here.
We suspect that the sets $\mathcal{T}_{a,b}$ we have constructed may be helpful in determining whether such a set $S$ can be constructed.
\end{section}

\subsection*{Acknowledgment}
		The authors would like to thank Han Peters for stimulating discussions. We are moreover grateful for constructive comments from the referees. We also thank Eoin Hurley for spotting a mistake in the proof of Lemma~\ref{def:uniqueness} in a previous version.
\appendix
\section{Proof of Lemma~\ref{def:uniqueness}}\label{app:proof}
We provide a proof of Lemma~\ref{def:uniqueness} here closely following Brightwell and Winkler's proof for the case $w=0$ modifying it where appropriate.

We start with the `if' part.
Fix $d$ and $w\geq 0$ and let $\mu$ be any Gibbs measure on $\mathbb{T}=\mathbb{T}_d$.
Let $U\subset V=V(\mathbb{T})$ be a finite set.
We aim to show that for any configuration $\psi:U\to [q]$, the probability
\begin{equation}
\PPr_\mu[{\bf \Phi}\!\restriction_U=\psi]
\end{equation}
does not depend on $\mu$.

We may assume that $U$ induces a tree with each vertex of degree $d+1$ or $1$ by taking a larger finite set if needed.
Suppose that $U$ has $\ell$ leaves; denote the set of leaves by $L$. For $n \in \mathbb{Z}_{\geq 1}$ let $W_n$ denote the collection of all vertices of $\mathbb{T}$ at distance at most $n$ from $U$. The graph induced by $(W_n\setminus U)\cup L$ is the disjoint union of $\ell$ copies of $\mathbb{T}^n$ each rooted at a leaf of $U$, we denote the tree rooted at $u \in L$ with $T_u$. We claim 
\begin{equation}\label{eq:factor prob1Pjotr}
   \lim_{n \to \infty} \max_{\rho_n: \partial W_n \to [q]} \left|\PPr_{\mu}[{\bf \Phi}\!\restriction_{U}=\psi \,  \big|\,  {\bf \Phi}\!\restriction_{\partial W_n}=\rho_n] - \PPr_U[{\bf \Psi}=\psi]\right| = 0,
\end{equation}
where $\bf \Psi$ is drawn from the Potts model distribution on $\mathbb{T}[U]$. This is sufficient because it follows that the difference
\begin{align*}
    \left|\PPr_\mu[{\bf \Phi}\!\restriction_U=\psi] - \PPr_U[{\bf \Psi}=\psi] \right| &= \left|  \sum_{\rho_n:\partial W_n\to [q]} \PPr_\mu[{\bf \Phi}\!\restriction_{\partial W_n}=\rho_n]\cdot \left( \PPr_{\mu}[{\bf \Phi}\!\restriction_{U}=\psi \,  \big|\,  {\bf \Phi}\!\restriction_{\partial W_n}=\rho_n] -  \PPr_U[{\bf \Psi}=\psi] \right) \right|\\
    &\leq  \max_{\rho_n:\partial W_n\to [q]} \left|\PPr_{\mu}[{\bf \Phi}\!\restriction_{U}=\psi \,  \big|\,  {\bf \Phi}\!\restriction_{\partial W_n}=\rho_n] - \PPr_U[{\bf \Psi}=\psi]\right|,
\end{align*}
can be made arbitrarily small, from which we conclude that $\PPr_\mu[{\bf \Phi}\!\restriction_U=\psi] = \PPr_U[{\bf \Psi}=\psi]$. As this does not depend on $\mu$ it shows that $\mu$ is unique. 

We now prove the claim. Let $\rho_n: \partial W_n \to [q]$ be arbitrary but fixed. Because $\mu$ satisfies the Gibbs property we see
\begin{equation}\label{eq:GibbsProperty}
\PPr_{\mu}[{\bf \Phi}\!\restriction_{U}=\psi \,  \big|\,  {\bf \Phi}\!\restriction_{\partial W_n}=\rho_n] = \PPr_{W_n}[{\bf \Phi'}\!\restriction_{U}=\psi \,  \big|\,  {\bf \Phi'}\!\restriction_{\partial W_n}=\rho_n],
\end{equation}
where ${\bf \Phi'}$ is drawn from the Potts model distribution on $\mathbb{T}[W_n]$.
We write $\phi \sim \psi $ if two configurations $\phi$ and $\psi$ are equal where they are both defined. Moreover, we denote the weight of a configuration $\sigma$ by $\weight(\sigma)$. By definition of the Potts model the right hand side of \eqref{eq:GibbsProperty} as
\begin{align*}
    \frac{\displaystyle \sum_{\substack{ \sigma:W_n \rightarrow [q]\\ \sigma\sim\psi,\ \sigma\sim \rho_n}} \weight(\sigma)}{\displaystyle \sum_{\substack{ \kappa:W_n \rightarrow [q]\\\kappa\sim \rho_n}} \weight(\kappa)}&= \frac{\weight(\psi)\displaystyle\sum_{\substack{ (\sigma_u)_{u\in L},\,\sigma_u:T_u \to [q]\\ \sigma_u\sim \rho_n,\ \sigma_u \sim \psi}} \prod_{u \in L} \weight(\sigma_u)}{\displaystyle\sum_{\kappa:U \rightarrow [q]} \weight(\kappa) \displaystyle\sum_{\substack{ (\gamma_u)_{u\in L},\,\gamma_u:T_u \to [q]\\ \gamma_u\sim \rho_n,\ \gamma_u \sim \kappa}} \prod_{u \in L} \weight(\gamma_u) }
     = \frac{\weight(\psi) \displaystyle\prod_{u \in L} \displaystyle \sum_{\substack{ \sigma_u:T_u \to [q]\\ \sigma_u\sim \rho_n,\,\sigma_u\sim \psi}}\weight(\sigma_u)}{\displaystyle\sum_{\kappa:U \rightarrow [q]} \weight(\kappa)  \prod_{u \in L} \sum_{\substack{ \gamma_u:T_u \rightarrow [q]\\ \gamma_u\sim \rho_n,\,\gamma_u\sim\kappa}}\weight(\gamma_u) }\\
    &= \frac{\weight{(\psi)}}{\sum_{\kappa:U \rightarrow [q]} \weight(\kappa)  \displaystyle \prod_{u \in L} \frac{\PPr_{T_u}[{\bf\Phi_u'}(u)=\kappa(u) \,  \big|\,  {\bf \Phi_u'}\!\restriction_{\partial T_u}=\rho_n\!\restriction_{\partial T_u}]}{\PPr_{T_u}[{\bf \Phi_u'}(u)=\psi(u) \,  \big|\,  {\bf \Phi_u'}\!\restriction_{\partial T_u}=\rho_n\!\restriction_{\partial T_u}]} },
\end{align*}
where $\partial T_u = T_u \cap \partial W_n$ and ${\bf \Phi_u'}$ is drawn from the the Potts model distribution on $\mathbb{T}[T_u]$.
As $n$ goes to infinity the distance between the root $u$ of $T_u$ and its leaves becomes arbitrarily large. It therefore follows from equation \eqref{equation:uniqueness} that the expression inside the final product gets arbitrarily close to $1$ uniformly over all $\rho_n$. We can thus conclude that $\PPr_{\mu}[{\bf \Phi}\!\restriction_{U}=\psi \,  \big|\,  {\bf \Phi}\!\restriction_{\partial W_n}=\rho_n]$ converges to 
\[
\frac{\weight(\psi)}{\sum_{\kappa: U \to [q]}\weight(\kappa)} = 
\PPr_U[{\bf \Psi}=\psi]
\]
uniformly, which was our claim.

For the `only if' part we merely sketch the argument.
Suppose the limsup is not equal to $0$ for some color $c\in [q]$. Then there must be distinct colors $c$ and $c'$, a number $\varepsilon>0$, a sequence $\{n_i\}$ of natural numbers and boundary conditions $\tau_i$ on the leaves of $\mathbb{T}^{n_i}_d$ such that the associated probabilities of the roots getting color $c$ (resp. $c'$) are at least $1/q+\varepsilon$ (resp. at most $1/q-\varepsilon$).
Let $\tau'_i$ be the boundary condition on the leaves of $\mathbb{T}^{n_i}_d$ obtained from $\tau_i$ by flipping the colors $c$ and $c'$. By symmetry, these respective probabilities are then reversed.
We can then create two distinct Gibbs measures with a limiting process using the boundary conditions $\tau_i$ and $\tau'_i$ respectively.

\printbibliography

\end{document}